%% file: bkequiv_compact.tex
\documentclass[12pt]{article}
\usepackage{amsmath, amsthm, amssymb}
\usepackage[top=1.0in, bottom=1.0in, left=1.0in, right=1.0in]{geometry}
\usepackage{hyperref}
\usepackage{graphicx}
\usepackage{tikz,tkz-graph,tkz-berge}
\usetikzlibrary{positioning}

\usetikzlibrary{arrows}
\usetikzlibrary{shapes}
\usepackage[position=bottom]{subfig}

\usepackage{longtable}
\usepackage{array}

\makeatletter
\newtheorem*{rep@theorem}{\rep@title}
\newcommand{\newreptheorem}[2]{
\newenvironment{rep#1}[1]{
 \def\rep@title{#2 \ref{##1}}
 \begin{rep@theorem}}
 {\end{rep@theorem}}}
\makeatother

\theoremstyle{plain}
\newtheorem{thm}{Theorem}[section]

\newtheorem{lem}[thm]{Lemma}
\newtheorem{conjecture}[thm]{Conjecture}
\newtheorem{cor}[thm]{Corollary}

\newtheorem*{SmallPotLemma}{Small Pot Lemma}

\newtheorem*{ClassificationOfd0}{Classification of $d_0$-choosable graphs}

\theoremstyle{definition}
\newtheorem{defn}{Definition}
\theoremstyle{remark}

\newcommand{\fancy}[1]{\mathcal{#1}}
\newcommand{\C}[1]{\fancy{C}_{#1}}
\newcommand{\IN}{\mathbb{N}}

\newcommand{\surj}{\twoheadrightarrow}

\newcommand{\set}[1]{\left\{ #1 \right\}}
\newcommand{\setb}[3]{\left\{ #1 \in #2 \mid #3 \right\}}
\newcommand{\setbs}[2]{\left\{ #1 \mid #2 \right\}}
\newcommand{\card}[1]{\left|#1\right|}

\newcommand{\floor}[1]{\left\lfloor#1\right\rfloor}
\newcommand{\func}[3]{#1\colon #2 \rightarrow #3}

\newcommand{\funcsurj}[3]{#1\colon #2 \surj #3}

\newcommand{\djunion}[2]{#1 \mbox{\hspace{2 pt}$+$\hspace{2 pt}} #2}
\newcommand{\parens}[1]{\left( #1 \right)}
\newcommand{\brackets}[1]{\left[ #1 \right]}
\newcommand{\DefinedAs}{\mathrel{\mathop:}=}
\newcommand{\join}[2]{#1 \mbox{\hspace{2 pt}$\vee$\hspace{2 pt}} #2}

\title{Coloring a graph with $\Delta-1$ colors: Conjectures equivalent to the
Borodin-Kostochka conjecture that appear weaker} 
\author{Daniel W. Cranston\thanks{Department of Mathematics and Applied Mathematics, Virginia Commonwealth University, Richmond, VA, 23284. email: \texttt{dcranston@vcu.edu}} \and Landon Rabern\thanks{School of Mathematical and Statistical Sciences, Arizona
State University, Tempe, AZ 85287.  email: \texttt{landon.rabern@gmail.com}. 
}
}

\begin{document}
\maketitle
\begin{abstract}
Borodin and Kostochka conjectured that every graph $G$ with maximum degree
${\Delta\ge9}$ satisfies $\chi\le \max\set{\omega, \Delta-1}$.
We carry out an in-depth study of minimum counterexamples to the
Borodin-Kostochka conjecture.  Our main tool is the identification of graph joins that are
$f$-choosable, where $f(v) \DefinedAs d(v) - 1$ for each vertex $v$.
Such a join cannot be an induced subgraph of a vertex critical graph with
$\chi = \Delta$, so we have a wealth of structural information about minimum
counterexamples to the Borodin-Kostochka conjecture.  

Our main result proves that certain conjectures that are prima facie weaker
than the Borodin-Kostochka conjecture are in fact equivalent to it.  One such
equivalent conjecture is the following:
Any graph with $\chi = \Delta = 9$ contains $\join{K_3}{\overline{K_6}}$ as a
subgraph.
\end{abstract}

\input{Introduction}
\input{Mules}
\input{ListColoringLemmas}

\section*{Acknowledgments}
We thank the two anonymous referees, whose numerous suggestions greatly improved
the presentation of our results.
\bibliographystyle{amsplain}
\bibliography{GraphColoring}
\end{document}

%% file: Introduction.tex
\section{Introduction}
\subsection{A short history of the problem}
We begin by briefly recounting the history of upper bounds on a graph's
chromatic number $\chi$ in terms of its maximum degree $\Delta$ and maximum
clique size $\omega$\footnote{Technically, the chromatic number, maximum degree,
and maximum clique size are parameters of a graph $G$; they are often denoted
$\chi(G)$, $\Delta(G)$, and $\omega(G)$.  However, when the graph $G$ is clear
from context, for brevity we simply write $\chi$, $\Delta$, and $\omega$.}.
The first non-trivial result about coloring graphs with around $\Delta$ colors
is Brooks' Theorem, from 1941. \begin{thm}[Brooks \cite{brooks1941colouring}]
Every graph with $\Delta \geq 3$ satisfies $\chi \leq \max\{\omega, \Delta\}$.
\end{thm}

In 1977, Borodin and Kostochka conjectured that a similar result holds for
$(\Delta - 1)$-colorings.  

\begin{conjecture}[Borodin and Kostochka \cite{borodin1977upper}]
Every graph with $\Delta \geq 9$ satisfies $\chi \leq \max\{\omega, \Delta - 1\}$.
\end{conjecture}

Constructions exist showing that the condition $\Delta \geq 9$ is
tight (see Figures \ref{fig:M_61}, \ref{fig:M_71},
\ref{fig:M_72} and \ref{fig:M_8}, starting in Section~\ref{excludingInMules}).
In the same paper where they posed the conjecture, Borodin and Kostochka proved
the following weakening. The proof is simple and uses a decomposition lemma of Lov\'{a}sz from the 1960s \cite{lovasz1966decomposition}.

\begin{thm}[Borodin and Kostochka \cite{borodin1977upper}]\label{BorodinKostochkaBK}
Every graph with $\chi = \Delta \geq 7$ contains $K_{\floor{\frac{\Delta + 1}{2}}}$.
\end{thm}

In the 1980s, Kostochka proved the following using a complicated recoloring
argument together with a technique for reducing $\Delta$ in a counterexample
based on hitting every maximum clique with an independent set (which we explain
following Lemma~\ref{HittingMaxCliques}).

\begin{thm}[Kostochka \cite{kostochkaRussian}]\label{KostochkaBK}
Every graph with $\chi = \Delta$ contains $K_{\Delta - 28}$.
\end{thm}

Kostochka \cite{kostochkaRussian} proved the following result which shows that
graphs having clique number sufficiently close to their maximum degree contain
an independent set hitting every maximum clique. In \cite{rabernhitting} the
second author improved the antecedent to $\omega \geq \frac34(\Delta + 1)$. 
Finally, King  \cite{KingHitting} made the result tight.

\begin{lem}[Kostochka \cite{kostochkaRussian}]
If $G$ is a graph with $\omega(G) \geq \Delta + \frac32 - \sqrt{\Delta}$,
then $G$ contains an independent set $I$ such that $\omega(G - I) < \omega(G)$.
\end{lem}
\begin{lem}[Rabern \cite{rabernhitting}]
If $G$ is a graph with $\omega(G) \geq \frac34 (\Delta + 1)$,
then $G$ contains an independent set $I$ such that $\omega(G - I) < \omega(G)$.
\end{lem}
\begin{lem}[King \cite{KingHitting}]\label{HittingMaxCliques}
If $G$ is a graph with $\omega(G) > \frac23 (\Delta + 1)$, then $G$ contains
an independent set $I$ such that $\omega(G - I) < \omega(G)$.
\end{lem}

If $G$ is a vertex critical graph with $\omega > \frac23 (\Delta + 1)$ and
we expand the independent set $I$ guaranteed by Lemma \ref{HittingMaxCliques}
to a
maximal independent set $M$ and remove $M$ from $G$, we see that $\Delta(G-M)
\leq \Delta(G) - 1$ and $\omega(G - M) = \omega(G) - 1$ and $\chi(G-M) =
\chi(G) - 1$; here $\chi(G-M)<\chi(G)$ since $G$ is vertex critical and
$\chi(G-M)\ge \chi(G)-1$ since $M$ is independent.
Using this approach, the proof of many coloring results can be reduced to the
case of the smallest $\Delta$ for which they hold.  In the case of graphs with
$\chi = \Delta$, we get the general result in Lemma~\ref{InductingOnC}.
First we need a definition and some notation.

\begin{defn}
For $k, j \in \IN$, let $\C{k, j}$ be the collection of all vertex critical
graphs with $\chi = \Delta = k$ and $\omega < k - j$.  Put $\C{k} \DefinedAs \C{k, 0}$. Note that $\C{k, j} \subseteq \C{k, i}$ for $j \geq i$.
\end{defn}

For each $k\ge 9$, the set $\C{k}$ is precisely the set of counterexamples to
the Borodin--Kostochka conjecture with $\Delta=k$.
For the following lemma, we need more notation.  We write $H \unlhd G$ if $H$ is
an induced subgraph of $G$, and $H \lhd G$ if $H$ is a \emph{proper} induced subgraph.

\begin{lem}\label{InductingOnC}
Fix $k, j \in \IN$ with $k \geq 3j + 6$.  If $G \in \C{k, j}$, then there exists $H \in \C{k-1, j}$
such that $H \lhd G$. 
\end{lem}
\begin{proof}
Let $G \in \C{k, j}$. We first show that there exists a maximal independent
set $M$ such that  $\omega(G - M) < k - (j + 1)$.   If $\omega(G)
< k - (j + 1)$, then any maximal independent set will do for $M$.
Otherwise, $\omega(G) = k - (j + 1)$.  Since $k \geq 3j + 6$, we have
$\omega(G) = k - (j + 1) \ge 2j+5 > \frac23(k + 1) = \frac23(\Delta(G) + 1)$. 
Thus by Lemma \ref{HittingMaxCliques}, we have an independent set $I$ such
that $\omega(G - I) < \omega(G)$.  Expand $I$ to a maximal independent set
to get $M$.

Now $\chi(G - M) = k - 1 = \Delta(G - M)$, where the last equality
follows from Brooks' Theorem. Further, $\omega(G - M) < k - (j + 1) \leq k - 1$.  Since $\omega(G - M) < k - (j + 1)$, for any $(k - 1)$-critical induced subgraph $H \unlhd G - M$ we have $H \in \C{k - 1, j}$.
\end{proof}

As a consequence we get the following result of
Kostochka~\cite{kostochkaRussian}, and also of Catlin~\cite{CatlinDiss}, that
the Borodin--Kostochka conjecture can be reduced to the case when $\Delta=k = 9$.

\begin{lem}[Kostochka~\cite{kostochkaRussian}, Catlin~\cite{CatlinDiss}]
\label{HereditaryReduction}
Let $\fancy{H}$ be a hereditary class of graphs (closed under deleting
vertices).  For $k \geq 5$, if $\fancy{H} \cap \C{k} = \emptyset$, then
$\fancy{H} \cap \C{k+1} = \emptyset$.  In particular, to prove the
Borodin--Kostochka conjecture it is enough to show that $\C{9} = \emptyset$.
\end{lem}

A little while after Kostochka proved his bound, Mozhan \cite{mozhan1983} proved
the following using a different technique.

\begin{thm}[Mozhan \cite{mozhan1983}]\label{MozhanTwoThirdsBK}
Every graph with $\chi = \Delta \geq 10$ contains $K_{\floor{\frac{2\Delta + 1}{3}}}$.
\end{thm}

In his dissertation Mozhan improved on this result.  We don't know the
method of proof as we were unable to obtain a copy of his dissertation. 
However, we suspect the method is a more complicated version of the proof of
Theorem \ref{MozhanTwoThirdsBK}.

\begin{thm}[Mozhan]\label{MozhanBK}
Every graph with $\chi = \Delta \geq 31$ contains $K_{\Delta - 3}$.
\end{thm}

Recently, we combined Mozhan's techniques with forbidding certain induced
subgraphs, using list-coloring results (see Section~\ref{ListColoringSection}
of this paper).
These methods allowed us to strengthen Mozhan's result to the following.

\begin{thm}[Cranston and Rabern~\cite{big-cliques}]\label{big-cliques}
Every graph with $\chi = \Delta \geq 13$ contains $K_{\Delta - 3}$.
\end{thm}

In 1999, Reed used probabilistic methods to prove that the Borodin--Kostochka
conjecture holds for graphs with sufficiently large maximum degree.  In his
paper, he proved that $\Delta\ge 10^{14}$ suffices.  However, he commented
that with a more detailed analysis of the argument, the hypothesis 
could probably be weakened to $\Delta\ge10^{6}$ and maybe even $\Delta\ge10^3$.

\begin{thm}[Reed \cite{reed1999strengthening}]\label{ReedBK}
Every graph satisfying $\chi = \Delta \geq 10^{14}$ contains $K_\Delta$.
\end{thm}

A lemma from Reed's proof of the above theorem is generally useful.

\begin{lem}[Reed \cite{reed1999strengthening}]\label{ReedsLemma}
Let $G$ be a vertex critical graph satisfying $\chi = \Delta \geq 9$ having the
minimum number of vertices. 
If $H$ is a $(\Delta - 1)$-clique in $G$, then any vertex in $G - H$ has at
most four neighbors in $H$.  In particular, the $(\Delta - 1)$-cliques in $G$ are pairwise disjoint.
\end{lem}

\subsection{Our contribution}
We carry out an in-depth study of minimum counterexamples to the
Borodin--Kostochka conjecture.  Our main tool is the exclusion of induced
subgraphs 
which are $f$-choosable, where $f(v) \DefinedAs d(v) - 1$ for each vertex
$v$.  (For definitions of $f$-choosable and related terms, see
Section~\ref{ListColoringSection}.)
Since an $f$-choosable graph cannot be an induced subgraph of a vertex
critical graph with $\chi = \Delta$, we have a wealth of structural information
about minimum counterexamples to the Borodin--Kostochka conjecture.  In Section
\ref{MulesSection}, we exploit this information and minimality to improve
Reed's Lemma \ref{ReedsLemma} as follows (see Corollary~\ref{AtMostOneEdgeIn}). 

\begin{lem}
Let $G$ be a vertex critical graph satisfying $\chi = \Delta \geq 9$ having the
minimum number of vertices.  If $H$ is a $(\Delta - 1)$-clique in $G$, then any
vertex in $G - H$ has at most one neighbor in $H$.
\end{lem}

Moreover, we lift the result out of the context of a minimum counterexample to
the Borodin--Kostochka conjecture, and view it in the more general setting of
graphs satisfying a certain criticality condition---we call such graphs mules.
This allows us to prove meaningful results for values of $\Delta$ less than $9$.  

Let $K_t$ and $E_t$ be the complete and edgeless graphs on $t$ vertices,
respectively.  
The disjoint union of graphs $G$ and $H$ is denoted $\djunion{G}{H}$.
The \emph{join} of graphs $G$ and $H$, denoted $\join{G}{H}$, is formed from
disjoint copies of $G$ and $H$ by adding every edge with one endpoint in $G$ and
one endpoint in $H$.
Since a graph containing $K_\Delta$ as a subgraph also
contains $\join{K_t}{E_{\Delta - t}}$ as a subgraph for any $t \in
\{1,\ldots,\Delta - 1\}$, the Borodin--Kostochka conjecture implies the
following conjecture.  Our main result 
(Corollary~\ref{MainResult})
is that the two conjectures are equivalent.

\begin{conjecture}\label{K3Conjecture}
Any graph with $\chi = \Delta \geq 9$ contains $\join{K_3}{E_{\Delta-3}}$ as a
subgraph.
\end{conjecture}

In fact, using Kostochka's reduction (Lemma~\ref{HereditaryReduction}) to the
case $\Delta = 9$, the following conjecture is also equivalent.

\begin{conjecture}\label{K3ConjectureReduced}
Any graph with $\chi = \Delta = 9$ contains $\join{K_3}{E_6}$ as a
subgraph.
\end{conjecture}

%% file: Mules.tex
\section{Mules and the Main Result}\label{MulesSection}
In this section we exclude more induced subgraphs in a minimum counterexample to
the Borodin--Kostochka conjecture than we can exclude purely using list coloring
techniques.  In fact, we lift these results out of the context of a minimum
counterexample and state them in the more general setting of graphs satisfying
a certain criticality condition defined in terms of the ordering given in
Definitions~\ref{epi} and \ref{child}.  The main result of
Section~\ref{MulesSection} (and, in fact, of the whole paper) is
Lemma~\ref{K3sOut}, which implies that Conjecture~\ref{K3Conjecture} is
equivalent to the Borodin-Kostochka Conjecture.

We should note that many of the results in Section~\ref{MulesSection} rely on
list-coloring results from Section~\ref{ListColoringSection}.  The
list-coloring results are independent of everything in Section
\ref{MulesSection}.  However, to focus attention on our main result, that
Conjecture~\ref{K3Conjecture} is equivalent to the Borodin-Kostochka
Conjecture, we first prove it in Section~\ref{MulesSection}, and only afterward
do we prove the list-coloring results in Section~\ref{ListColoringSection}.

\begin{defn}
\label{epi}
If $G$ and $H$ are graphs, an \emph{epimorphism} is a graph homomorphism $\funcsurj{f}{G}{H}$ such that $f(V(G)) = V(H)$.  We indicate this with the arrow $\surj$.
\end{defn}

\begin{defn}
\label{child}
Let $G$ be a graph.  A graph $A$ is called a \emph{child} of $G$ if $A \neq G$
and there exists $H \unlhd G$ and an epimorphism $\funcsurj{f}{H}{A}$.  
(Recall that $H \unlhd G$ denotes that $H$ is an induced subgraph of $G$.)
\end{defn}

Note that the child-of relation is a strict partial order on the set of (finite
simple) graphs $\fancy{G}$.  We call this the \emph{child order} on $\fancy{G}$
and denote it by `$\prec$'.  By definition, if $H \lhd G$ then $H \prec G$.

\begin{lem}\label{well-founded}
The ordering $\prec$ is well-founded on $\fancy{G}$; that is, every nonempty
subset of $\fancy{G}$ has a minimal element under $\prec$.
\end{lem}
\begin{proof}
Let $\fancy{T}$ be a nonempty subset of $\fancy{G}$.  Pick $G \in \fancy{T}$
minimizing $\card{V(G)}$ and then maximizing $\card{E(G)}$.  
Since any child of $G$ must have fewer vertices or more edges (or both), we see
that $G$ is minimal in $\fancy{T}$ with respect to $\prec$.
\end{proof}

\begin{defn}
Let $\fancy{T}$ be a collection of graphs.  A minimal graph in $\fancy{T}$
under the child order is called a \emph{$\fancy{T}$-mule}.
\end{defn}

With the definition of mule we have captured the important properties (for
coloring) of a counterexample first minimizing the number of vertices and then
maximizing the number of edges.  Viewing $\fancy{T}$ as a set of
counterexamples, we can add edges to or contract independent sets in induced
subgraphs of a $\fancy{T}$-mule and get a non-counterexample.  We could do the
same with a minimum counterexample, but with mules we have more minimal objects
to work with.  One striking consequence of this is that many of our proofs
naturally construct multiple counterexamples to the Borodin-Kostochka
Conjecture for small $\Delta$.  The small counterexamples $M_{6,1}$ (Figure
\ref{fig:M_61})  and $M_{7,1}$ (Figure \ref{fig:M_7}) were constructed in 1978
by Benedict and Chinn
\cite{BenedictAndChinn1978}.  In his dissertation \cite{CatlinDiss}, Catlin
extended this construction to get an infinite family of counterexamples for
$\Delta = 6$ and seven counterexamples for $\Delta=7$.  This construction uses a
special case of gadgets that Molloy and Reed \cite{molloy2002graph} call
\emph{reducers}.  These reducer gadgets were used by Emden-Weinert, Hougardy,
and Kreuter \cite{emden1998uniquely} in their proof that $(\Delta + 1 -
k)$-coloring is NP-hard when $k^2 + k > \Delta$.  For $\Delta=8$, the one known
counterexample $M_8$ (Figure \ref{fig:M_8}) was constructed by Catlin
\cite{catlin1979hajos} as a counterexample to the Haj{\'o}s conjecture.  Catlin
did not know of this counterexample at the time of his dissertation and the
same is likely true for $M_{7,2}$ (Figure \ref{fig:M_72}) which is $M_8$ with a
maximum independent set removed.

The more general mule-framework puts these different examples into context and
gives us some understanding of why they should be counterexamples, instead of
just knowing that they are.  We are hopeful that a better understanding of why
the Borodin-Kostochka conjecture fails for small $\Delta$ will contribute to its
solution for large $\Delta$.  In \cite{king2012fractional}, King, Lu and Peng
prove the Borodin-Kostochka conjecture for the fractional chromatic number by
reducing a more general statement to the $\Delta = 4$ case.  We think it may be
possible to do something similar for the Borodin-Kostochka conjecture, but we'd
need a better understanding of counterexamples for small $\Delta$.  For example,
our main result in Corollary \ref{MainResult} shows that to prove the
Borodin-Kostochka conjecture, it suffices to find a $\join{K_3}{E_{\Delta-3}}$
subgraph, instead of the (larger) subgraph $K_\Delta$.   Since all counterexamples
we know of for $\Delta \ge 6$ contain a $\join{K_3}{E_{\Delta-3}}$ subgraph, it
may be possible to prove this more general conjecture for $\Delta=6$ and then extend that up to large $\Delta$.

\subsection{Excluding induced subgraphs in mules}
\label{excludingInMules}
Our main goal in this section is to prove Lemma~\ref{K4sOut}, which says that
(with only one exception) for $k\ge 7$, no $\C{k}$-mule contains
$\join{K_4}{E_{k-4}}$ as a subgraph.  This result immediately implies that the
Borodin-Kostochka Conjecture is equivalent to Conjecture~\ref{K4Conjecture}. 
This equivalence is a major step toward our main result.  Our approach is
based on Lemma~\ref{K_tClassification}, which implies that if $G$ is a
counterexample to Lemma~\ref{K4sOut}, then the vertices of the $E_{k-4}$ induce
either $E_3$, a claw ($K_{1,3}$), a clique, or an almost complete graph (a graph $H$ is
\emph{almost complete} if $\omega(H)\ge |V(H)|-1$).  Our job in this section
consists of showing that each of these four possibilities is, in fact,
impossible.  Ruling out the clique is easy, but the other possibilities require
more work.  The cases of $E_3$ and the claw are handled in Lemma~\ref{NoE3},
and the case of an almost complete graph (which requires the most work) is
handled by Corollary~\ref{AtMostOneEdgeIn}.  The rest of the section consists
of structural lemmas, which build toward the proofs of Lemmas~\ref{NoE3} and
\ref{NoForksThatArentKnives} (which in turn yields
Corollary~\ref{AtMostOneEdgeIn}).  
We give names to a few small graphs.  The \emph{chair} is formed from $K_{1,3}$
by subdividing one edge.  The \emph{antichair} is its complement.
The \emph{paw} is formed from $K_3$ by adding a pendant edge at one vertex.  The
\emph{antipaw} is its complement.
In Sections~\ref{JoinsWithThree} and \ref{JoinsWithTwos}, we depict three of
these graphs in Figures~\ref{fig:antipaw} and \ref{fig:chair_antichair}.
\bigskip

For $k \in \IN$, by a \emph{$k$-mule} we mean a $\C{k}$-mule.

\begin{lem}\label{EpiPower}
Let $G$ be a $k$-mule with $k \geq 4$.  If $A$ is a child of $G$ with $\Delta(A) \leq k$ then either
\begin{itemize}
\item $A$ is $(k - 1)$-colorable; or
\item $A$ contains $K_k$.
\end{itemize}
\end{lem}
\begin{proof}
Let $A$ be a child of $G$ with $\Delta(A) \le k$.  Suppose $\chi(A) \ge k$ and
$\omega(A) < k$.  By Brooks' Theorem, $\chi(A) = \Delta(A) = k$.  Since $A
\prec G$ and $G$ is a mule, $A \not \in \C{k}$ and thus $A$ is not vertex
critical.  Let $A'$ be a vertex critical induced subgraph of $A$; observe that
$A' \prec G$.  Thus $\chi(A') = k$ and $\omega(A') < k$, so $\Delta(A') = k$.
Consequently $A' \in \C{k}$, contradicting the fact that $G$ is a mule.
\end{proof}

Adding edges to a graph yields an epimorphism; specifically, if $H$ is a
spanning subgraph of $G$, we have the \emph{inclusion epimorphism}
$\funcsurj{f}{H}{G}$ given by $f(v) = v$ for all $v \in V(H)$.
In the next lemma and thereafter, we write $d_H(v)$, for a vertex $v$ and a
subgraph $H$, to denote $|N(v)\cap V(H)|$.

\begin{lem}\label{UnequalColoredPairOrCliqueMinusEdge}
Let $G$ be a $k$-mule with $k \geq 4$ and $H \unlhd G$.  
Assume $x, y \in V(H)$, $xy \not \in E(H)$ and both $d_H(x) \leq k-1$ and
$d_H(y) \leq k-1$.  If for every $(k - 1)$-coloring $\pi$ of $H$ we have
$\pi(x) = \pi(y)$, then $H$ contains $\join{\set{x, y}}{K_{k-2}}$.
\end{lem}
\begin{proof}
Suppose that for every $(k - 1)$-coloring $\pi$ of $H$ we have $\pi(x) =
\pi(y)$.  Using the inclusion epimorphism $\funcsurj{f_{xy}}{H}{H + xy}$ in
Lemma \ref{EpiPower} shows that either $H + xy$ is $(k - 1)$-colorable or $H +
xy$ contains $K_k$.  Since a $(k - 1)$-coloring of $H + xy$ would induce a
$(k - 1)$-coloring of $H$ with $x$ and $y$ colored differently, we conclude
that $H + xy$ contains $K_k$.  But then $H$ contains $\join{\set{x,
y}}{K_{k-2}}$ and the proof is complete.
\end{proof}

We will often begin by coloring some subgraph $H$ of our graph $G$, and work to
extend this partial coloring.  More formally, let $G$ be a graph and $H \lhd
G$.  For $t \geq \chi(H)$, let $\pi$ be a proper $t$-coloring of $H$.  
For each $x \in V(G-H)$, put $L_{\pi}(x) \DefinedAs \set{1, \ldots, t} -
\{\pi(y): y\in N(x)\cap V(H)\}$.
Then $\pi$ is extendable to a
$t$-coloring of $G$ if and only if $L_{\pi}$ admits a coloring of $G-H$. 
We will use this fact repeatedly in the proofs that follow.  The following
generalizes a lemma due to Reed \cite{reed1999strengthening}; the proof is
essentially the same.

\begin{lem}\label{E2impliesE3}
For $k \geq 6$, if a $k$-mule $G$ contains an induced copy of $\join{E_2}{K_{k -
2}}$, then $G$ contains an induced copy of $\join{E_3}{K_{k - 2}}$.
\end{lem}
\begin{proof}
Suppose $G$ is a $k$-mule containing an induced copy of $\join{E_2}{K_{k - 2}}$, call it $F$.  
Let $\{x_1, y_1\}$ be the set of vertices of degree $k-2$ in $F$ and let $C \DefinedAs \set{w_1,
\ldots, w_{k-2}}$ be the set of vertices of degree $k-1$ in $F$.  Let $H \DefinedAs G -
F$. Since $G$ is vertex critical, we can $(k-1)$-color $H$.  Doing so leaves a
list assignment $L$ on $F$ with $\card{L(z)} \geq d_F(z) - 1$ for each $z \in
V(F)$; as defined above, this is $L_{\pi}$, where $t=k-1$.
Now $\card{L(x_1)} + \card{L(y_1)} \geq d_F(x_1) + d_F(y_1) - 2 = 2k - 6 > k -
1$, since $k \geq 6$.  Hence we have $c \in L(x_1) \cap L(y_1)$.  
Coloring both $x_1$ and $y_1$ with $c$ leaves a list assignment $L'$ on $C$
with $\card{L'(w_i)} \geq k - 3$ for each $1 \leq i \leq k-2$.  
Now, if $\card{L'(w_i)} \geq k - 2$ for some $i$ or if $L'(w_i) \neq L'(w_j)$
for some $i$ and $j$, then we can complete the partial $(k - 1)$-coloring to
all of $G$ using Hall's Theorem.  Hence we must have $d(w_i) = k$ and $L'(w_i)
= L'(w_j)$ for all $i$ and $j$.  Let $N \DefinedAs \bigcup_{w \in C} N(w) \cap
V(H)$ and note that $N$ is an independent set since it is contained in a single
color class in every $(k - 1)$-coloring of $H$. Also, each $w \in C$ has
exactly one neighbor in $N$.

Proving that $\card{N} = 1$ will give the desired induced $\join{E_3}{K_{k -
2}}$ in $G$; if instead this subgraph is not induced, then the vertex in $N$ is
adjacent to $x_1$ or $y_1$, so $G$ contains $K_k$, and is not a mule.  Thus, to
reach a contradiction, suppose that $\card{N} \geq 2$.  

We will repeatedly find copies of $\join{E_2}{K_{k-2}}$ in $G$; when $x$ and $y$
are the vertices in $E_2$, we will write $K(x,y)$ to denote the corresponding
copy of $K_{k-2}$.
We know that $H$ has no $(k - 1)$-coloring in which two vertices of $N$ get
different colors, since then we could complete the partial coloring as above. 
Let $x_2, y_2 \in N$ be distinct. Since both $x_2$ and $y_2$ have a neighbor
in $F$, we may apply Lemma \ref{UnequalColoredPairOrCliqueMinusEdge} to conclude 
that $\join{\set{x_2, y_2}}{K(x_2, y_2)}$ is in $H$. 

First, suppose $\card{N} \geq 3$; say $\set{x_2, y_2, z_2}\subseteq N$.  We
have $w \in K(x_2, y_2) \cap K(x_2, z_2)$ for otherwise $d(x_2) \geq 2(k - 2)
> k$.  Since $w$ already has $k$ neighbors among $K(x_2, y_2) - \set{w}$ and
$x_2, y_2, z_2$, we must have $K(x_2, z_2) = K(x_2, y_2)$.  But then
$\set{x_2, y_2, z_2} \cup V(K(x_2, y_2))$ induces our desired $\join{E_3}{K_{k
- 2}}$ in $G$.

Hence we must have $\card{N} = 2$, say $N = \set{x_2, y_2}$.  Now each of $x_2$
and $y_2$ has $k - 2$ neighbors in $K(x_2, y_2)$ and thus at most two neighbors
in $C$.  Hence $\card{C} \leq 4$.  Thus, we must have $k = 6$.

We may apply the same reasoning to $\join{\set{x_2, y_2}}{K(x_2, y_2)}$ that we
did to $F$ to get vertices $x_3, y_3$ 
such that $\join{\set{x_3, y_3}}{K(x_3, y_3)}$ is in $G$.  But then we may do it again with $\join{\set{x_3, y_3}}{K(x_3, y_3)}$ and so on.  
More formally, we find a sequence of nonadjacent pairs of vertices $\{x_1,y_1\},
\{x_2,y_2\},\ldots$; for each pair $\{x_i,y_i\}$ we
find $K(x_i,y_i)$, a $K_4$ that is joined to $\{x_i,y_i\}$, such that
$x_{i+1}$ and $y_{i+1}$ each have two neighbors in $K(x_i,y_i)$.
Since $G$ is finite, this process must terminate, and therefore
$\{x_1,y_1\}=\{x_m,y_m\}$ for some $m\ge 2$. 
This graph is $5$-colorable since $\{x_1,y_1,x_2,y_2,\ldots,x_m,y_m\}$ can be
colored with one color and $\bigcup_{i=1}^{m-1}K(x_i,y_i)$ can be 4-colored.
This final contradiction completes the proof.
\end{proof}

The proof of the following lemma forces the construction of counterexamples to
the Borodin--Kostochka conjecture for $\Delta = 6$ and $\Delta = 7$.  The
first of these is $M_{6,1}$ (Figure \ref{fig:M_61}), which is created from the
disjoint union of $\join{K_4}{E_3}$ and $K_5$ by adding six edges between the
$E_3$ and $K_5$ such that each vertex in the $E_3$ is incident to exactly two
of the edges and each vertex in the $K_5$ is incident to at least one
of them.  Note that we get another counterexample (not shown) for $\Delta = 6$
if we instead add only five edges, but this counterexample is not a mule
(precisely since we can add an edge and get another counterexample). The second
counterexample is $M_{7,1}$ (Figure \ref{fig:M_7}) which is created from the
disjoint union of $\join{K_5}{E_3}$ and $K_6$ by adding six edges between the
$E_3$ and $K_6$ such that each vertex in the $E_3$ is incident to exactly two
of the edges and each vertex in the $K_6$ is incident to exactly one of the
edges.
\input{M_61}
\input{M7}

\begin{lem}\label{NoE2}
For $k \geq 6$, the only $k$-mules containing an induced copy of $\join{E_2}{K_{k - 2}}$
are $M_{6,1}$ and $M_{7,1}$ (see Figures~\ref{fig:M_61} and \ref{fig:M_7}).
\end{lem}
\begin{proof}
Suppose we have a $k$-mule $G$ that contains an induced copy of $\join{E_2}{K_{k - 2}}$. 
Then by Lemma \ref{E2impliesE3}, $G$ contains an induced copy of $\join{E_3}{K_{k - 2}}$, call it $F$. 

Let $x, y, z$ be the vertices of degree $k-2$ in
$F$ and let $C \DefinedAs \set{w_1, \ldots, w_{k-2}}$ be the set of vertices of degree
$k$ in $F$. Put $H \DefinedAs G - C$.  
Let $A$ be the graph formed from $H$ by contracting $\set{x,y,z}$ to a single vertex $v_{xyz}$.
Since each of $x$, $y$, and $z$ have degree at most $2$ in $H$, we have $\Delta(A) \le
\max(6,k) \le k$.  We have an epimorphism $\funcsurj{f}{H}{A}$ given by $f(v) = v$ for $v
\not \in \set{x,y,z}$ and $f(v) = v_{xyz}$ for $v \in \set{x,y,z}$ and hence $A
\prec G$.  But $G$ is a mule, so $\chi(A) \le k-1$ or $\omega(A) \ge k$ by Lemma
\ref{EpiPower}.  If $\chi(A) \le k-1$, then we have a $(k-1)$-coloring of $H$
where $x$, $y$, and $z$ all get the same color, which we can greedily extend to
$C$, a contradiction.  So, we must have $\omega(A) \ge k$.  

Since $\omega(G) \le k-1$, the $k$-clique in A must contain vertex $v_{xyz}$,
since it is the only vertex of $A$ not appearing in $G$.  However $d(xyz) \le 6$.  
So $k \le 7$.
%
Moreover, $H$ contains $K_{k-1}$ (call it $D$) such
that $V(D) \subseteq N(x) \cup N(y) \cup N(z)$.  Put $Q \DefinedAs
G\brackets{V(F) \cup V(D)}$.  Then $Q$ is
$k$-chromatic and as $G$ is vertex critical, we must have $G = Q$.  If $k = 7$, then $G
= M_{7,1}$.  Suppose $k=6$ and $G \neq M_{6,1}$. Then one of $x$, $y$, or $z$
has only one neighbor in $D$.  By symmetry we may assume it is $x$.  But we can
add an edge from $x$ to a vertex in $D$ to form $M_{6,1}$ and hence $G$ has a
proper child, which is impossible.
\end{proof}

\begin{lem}\label{UnequalColoredPair}
Let $G$ be a $k$-mule with $k \geq 6$ other than $M_{6,1}$ and $M_{7,1}$ and let
$H \lhd G$. If $x, y \in V(H)$ and both $d_H(x) \leq k-1$ and $d_H(y) \leq k-1$, then there exists a $(k - 1)$-coloring $\pi$ of $H$ such that $\pi(x) \neq \pi(y)$.
\end{lem}
\begin{proof}
Suppose $x, y \in V(H)$ and both $d_H(x) \leq k-1$ and $d_H(y) \leq k-1$.  
First, if $xy \in E(H)$ then any $(k - 1)$-coloring of $H$ will do.  
Otherwise, if for every $(k - 1)$-coloring $\pi$ of $H$ we have $\pi(x) = \pi(y)$, then by Lemma \ref{UnequalColoredPairOrCliqueMinusEdge}, 
$H$ contains $\join{\set{x, y}}{K_{k-2}}$.  The lemma follows since this is impossible by Lemma \ref{NoE2}.
\end{proof}

\begin{lem}\label{JoinerOrDifferentLists}
Let $G$ be a $k$-mule with $k \geq 6$ other than $M_{6,1}$ and $M_{7,1}$ and let $F \lhd G$.  
Put $C \DefinedAs \setb{v}{V(F)}{d(v) - d_F(v) \leq 1}$.  At least one of the following holds:
\begin{itemize}
\item $G - F$ has a $(k - 1)$-coloring $\pi$ such that for some $x, y \in C$ we have $L_{\pi}(x) \neq L_{\pi}(y)$; or
\item $G - F$ has a $(k - 1)$-coloring $\pi$ such that for some $x \in C$ we have $\card{L_{\pi}(x)} = k - 1$; or
\item there exists $z \in V(G - F)$ such that $C \subseteq N(z)$.
\end{itemize}
\end{lem}
\begin{proof}
Put $H \DefinedAs G - F$.  
Suppose that for every $(k - 1)$-coloring $\pi$ of $H$ we have $L_{\pi}(x) = L_{\pi}(y)$ for every $x, y \in C$.  
By assumption, the vertices in $C$ have at most one neighbor in $H$.  
If some $v \in C$ has no neighbors in $H$, then for any $(k - 1)$-coloring $\pi$ of $H$ we have $\card{L_{\pi}(v)} = k - 1$.  
Thus we may assume that every $v \in C$ has exactly one neighbor in $H$. 

Let $N \DefinedAs \bigcup_{w \in C} N(w) \cap V(H)$. Suppose $\card{N} \geq 2$.
Pick different $z_1, z_2 \in N$. Then, by Lemma \ref{UnequalColoredPair}, there is a $(k - 1)$-coloring $\pi$ of $H$ for which $\pi(z_1) \neq \pi(z_2)$.  
But then $L_{\pi}(x) \neq L_{\pi}(y)$ for some $x, y \in C$ giving a contradiction.  Hence $N = \set{z}$ and thus $C \subseteq N(z)$.
\end{proof}

What follows next depends on Lemma~\ref{K_tClassification}, so we need another
definition used there.  (Recall also that a graph $H$ is \emph{almost complete}
if $\omega(H)\ge |V(H)|-1$.)  A graph is $d_1$-choosable if it has a proper
$L$-coloring for every list assignment $L$ such that $\card{L(v)}=d(v)-1$.
At the start of Section \ref{ListColoringSection}, we define more terms related
to list-coloring.
By Lemma \ref{K_tClassification}, no graph in $\C{k}$ contains an induced copy of
$\join{E_3}{K_{k - 3}}$ for $k \geq 9$.  
For mules, we can improve this as follows.

\begin{lem}\label{NoE3}
For $k \geq 7$, the only $k$-mule containing an induced copy of $\join{E_3}{K_{k - 3}}$ is $M_{7,1}$.
\end{lem}
\begin{proof}
Suppose the lemma is false and let $G$ be a $k$-mule, other than $M_{7,1}$, containing such an induced subgraph $F$.  
Let $z_1, z_2, z_3 \in F$ be the vertices with degree $k-3$ in $F$ and $C$ the rest of the vertices in $F$ (all of degree $k-1$ in $F$). 
Put $H \DefinedAs G - F$.

First suppose there is not a vertex $x \in V(H)$ which is adjacent to all of $C$. 
Let $\pi$ be a $(k - 1)$-coloring of $H$ guaranteed by Lemma
\ref{JoinerOrDifferentLists} and put $L \DefinedAs L_\pi$.  Since $\card{L(z_1)}
+ \card{L(z_2)} + \card{L(z_3)} \geq 3(k-4) > k - 1$ we have $1 \leq i < j \leq 3$ such that $L(z_i) \cap L(z_j) \neq \emptyset$.  
Without loss of generality, $i = 1$ and $j = 2$. Pick $c \in L(z_1) \cap L(z_2)$ and color both $z_1$ and $z_2$ with $c$.  
Let $L'$ be the resulting list assignment on $F - \set{z_1, z_2}$.  
Now $\card{L'(z_3)} \geq k-4$ and $\card{L'(v)} \geq k-3$ for each $v \in C$.  
By our choice of $\pi$, either two of the lists in $C$ differ or for some $v \in C$ we have $\card{L'(v)} \geq k-2$.  
In either case, we can complete the $(k - 1)$-coloring to all of $G$ by Hall's Theorem.

Hence we must have $x \in V(H)$ which is adjacent to all of $C$.  
Thus $G$ contains the induced subgraph $\join{K_{k-3}}{G[z_1, z_2, z_3, x]}$.  
Therefore $k = 7$ and $x$ is adjacent to each of $z_1, z_2, z_3$ by Lemma \ref{K_tClassification}.  
Hence $G$ contains the induced subgraph $\join{K_5}{E_3}$ contradicting Lemma \ref{NoE2}.
\end{proof}

\begin{lem}\label{NoTwooks}
For $k \geq 7$, no $k$-mule contains an induced copy of $\join{\overline{P_3}}{K_{k - 3}}$.
\end{lem}
\begin{proof}
Suppose the lemma is false and let $G$ be a $k$-mule containing such an induced
subgraph $F$.  Note that $M_{7,1}$ has no induced $\join{\overline{P_3}}{K_{k -
3}}$, so $G \neq M_{7,1}$. Let $z \in V(F)$ be the vertex with degree $k-3$ in
$F$, $v_1, v_2 \in F$ the vertices of degree $k-2$ in $F$ and $C$ the rest of the vertices in $F$ (all of degree $k-1$ in $F$). 
Put $H \DefinedAs G - F$.

First suppose there is not a vertex $x \in V(H)$ which is adjacent to all of $C$. 
Let $\pi$ be a $(k - 1)$-coloring of $H$ guaranteed by Lemma
\ref{JoinerOrDifferentLists} and put $L \DefinedAs L_\pi$.  Then, we have
$\card{L(z)} \geq k-4$ and $\card{L(v_1)} \geq k-3$.  
Since $k \geq 7$, $\card{L(z)} + \card{L(v_1)} \geq 2k - 7 > k - 1$.  
Hence, by the pigeonhole principle, we may color $z$ and $v_1$ the same.  
Let $L'$ be the resulting list assignment on $F - \set{z, v_1}$. 
Now $\card{L'(v_2)} \geq k-4$ and $\card{L'(v)} \geq k-3$ for each $v \in C$. 
By our choice of $\pi$, either two of the lists in $C$ differ or for some $v \in C$ we have $\card{L'(v)} \geq k-2$.  
In either case, we can complete the $(k - 1)$-coloring to all of $G$ by Hall's Theorem.

Hence we must have $x \in V(H)$ which is adjacent to all of $C$.
Thus $G$ contains the induced subgraph $\join{K_4}{G[z, v_1, v_2, x]}$.  
By Lemma \ref{K_tClassification}, $G[z, v_1, v_2, x]$ must be almost complete and hence $x$ must be adjacent to both $v_1$ and $v_2$.  
But then $\join{G[v_1, v_2, x]}{C}$ is a $K_k$ in $G$, giving a contradiction.
\end{proof}

Reed proved that for $k \geq 9$, a vertex outside a $(k - 1)$-clique $H$ in a
$k$-mule can have at most $4$ neighbors in $H$.  We improve this to at most one
neighbor.

\begin{figure}[h!tb]
\begin{center}
\includegraphics[scale=.75,trim=2cm 18.5cm 2cm 2.5cm, clip]{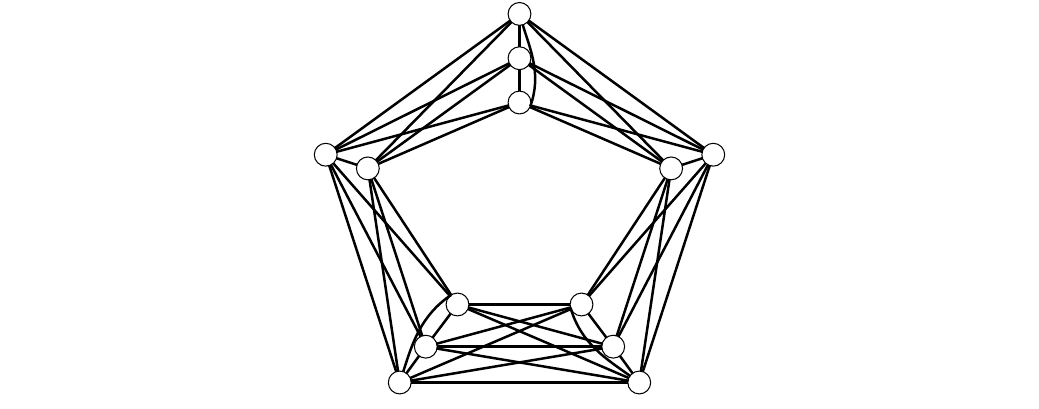}
\end{center}
\caption{The mule $M_{7,2}$, where $M_{7,2}=C_5\boxtimes K_3-2v$.
(Here $2v$ denotes a pair of nonadjacent vertices, and $\boxtimes$ denotes
the \emph{strong product}, such that $K_2\boxtimes K_2= K_4$.)}
\label{fig:M_72}
\end{figure}

\begin{lem}\label{NoForksThatArentKnives}
For $k \geq 7$ and $r \geq 2$, no $k$-mule except $M_{7,1}$ and $M_{7,2}$
(see Figures \ref{fig:M_7} and \ref{fig:M_72})
contains an induced copy of $\join{K_r}{\left(\djunion{K_1}{K_{k - (r + 1)}}\right)}$.
\end{lem}
\begin{proof}
Suppose the lemma is false and let $G$ be a $k$-mule, other than $M_{7,1}$ and
$M_{7,2}$, containing such an induced subgraph $F$ with $r$ maximal. By Lemma
\ref{NoE2} and Lemma \ref{NoTwooks}, the lemma holds for $r \geq k - 3$. So we have $r \leq k - 4$. Now, let $z \in V(F)$ be the vertex with degree $r$
in $F$, $v_1, v_2, \ldots, v_{k - (r + 1)} \in V(F)$ the vertices of degree $k -
2$ in $F$ and $C$ the rest of the vertices in $F$ (all of degree $k-1$ in $F$). Put $H \DefinedAs G - F$ and $H' \DefinedAs G\brackets{V(H) \cup \set{z}}$.

Let $Z_1 \DefinedAs \setbs{za}{a \in N(v_1) \cap V(H)}$.  
Consider the graph $D \DefinedAs H' + Z_1$.  
Since $v_1$ has at most two neighbors in $H$, $\card{Z_1} \leq 2$ and thus to
form $D$ from $H'$, we added $E(A)$ where $A \in \set{K_1, K_2, P_3, 2K_2}$.  Since
$\card{C} \geq 2$, $\Delta(D) \leq k$. Hence Lemma \ref{EpiPower} shows that $H
+ z$ contains $K_k - E(A)$ or $\chi(D) \leq k - 1$.  
Suppose $\chi(D) \geq k$. 
If $A = K_1$, $A = K_2$, or $A = P_3$, then 
we have a contradiction by the fact that $\omega(G) < k$, Lemma \ref{NoE2}, and
Lemma \ref{NoTwooks}, respectively.  If $A=2K_2$, then we get a
contradiction from Lemma~\ref{E2Classification}.
Thus
we must have $\chi(D) \leq k - 1$, which gives a $(k - 1)$-coloring of $H'$
in which $z$ receives a color $c$ which is not received by any of the neighbors
of $v_1$ in $H$. Thus $c$ remains in the list of $v_1$ and we may color $v_1$
with $c$.  After doing so, each vertex in $C$ has a list of size at least $k-3$
and $v_i$ for $i > 1$ has a list of size at least $k-4$.  If any pair of
vertices in $C$ had different lists, then we could complete the partial
coloring by Hall's Theorem.  Let $N \DefinedAs \bigcup_{w \in C} N(w) \cap
V(H)$ and note that $N$ is an independent set since it is contained in a single
color class in the $(k - 1)$-coloring of $H$ just constructed.

Suppose $\card{N} \geq 2$.  Pick $a_1, a_2 \in N$. 
Consider the graph $D \DefinedAs H' + Z_1 + a_1a_2$.  
Plainly, $\Delta(D) \leq k$. 
To form $D$ from $H'$ we added $E(A)$, where $A \in \set{K_2, P_3, K_3, 2K_2,
P_4, \djunion{K_2}{P_3}}$.  Hence Lemma \ref{EpiPower} shows that either
$H'$ contains $K_k - E(A)$ or else $\chi(D) \leq k - 1$.  
We show that the former case is impossible.
To show that $A=K_3$, $A=P_4$, $A=2K_2$ and $A=\djunion{K_2}{P_3}$ are impossible,
we apply Lemma \ref{NoE3} (this is where we use the fact that
$G\ne M_{7,1}$),  Lemma~\ref{E2Classification}, Lemma \ref{AJoinP_4} (since $K_t-E(P_4)=\join{P_4}{K_{t-4}}$),
and Lemma \ref{E2Classification} (using $\set{a_1, a_2}$ as the $E_2$), respectively.

Thus we must have $\chi(D) \leq k - 1$, which gives a $(k -
1)$-coloring of $H'$ in which $a_1$ and $a_2$ are in different color classes
and $z$ receives a color not received by any neighbor of $v_1$ in $H$.  As
above we can complete this partial coloring to all of $G$ by first coloring $z$
and $v_1$ the same and then using Hall's Theorem.

Hence there is a vertex $x \in V(H)$ which is adjacent to all of $C$.  
Note that $x$ is not adjacent to any of $v_1, v_2, \ldots, v_{k - (r + 1)}$ by the maximality of $r$. 
Let $Z_2 \DefinedAs \setbs{xa}{a \in N(v_2) \cap V(H)}$.  Consider the graph 
$D \DefinedAs H' + Z_1 + Z_2$.  As above, both $Z_1$ and $Z_2$ have
cardinality at most $2$.  Since $\card{C} \geq 2$, both $x$ and $z$ have degree at most $k$ in $D$.  
Since both $xa$ and $za$ were added only if $a$ was a neighbor of both $v_1$ and $v_2$, 
all the neighbors of $v_1$ in $H$ have degree at most $k$ in $D$. Similarly for $v_2$'s neighbors.  
Hence $\Delta(D) \leq k$. 
To form $D$ from $H'$ we added $E(A)$ where $A
\in \set{K_1, K_2, P_3, K_3, P_4, \djunion{K_2}{P_3}, 2K_2, P_5, 2P_3, C_4}$. 
Hence Lemma \ref{EpiPower} shows that $H'$ contains $K_k - E(A)$ or $\chi(D) \leq k - 1$.  

Suppose $\chi(D) \geq k$. Then $A = K_1$, $A = K_2$, $A = P_3$, $A =
K_3$, $A = P_4$, and $A = \djunion{K_2}{P_3}$ are impossible as above.  
Applying Lemma \ref{E2Classification} shows that $A = 2K_2$, $A = P_5$, and 
$A = 2P_3$ are impossible.  Thus we must have $A = C_4$.  If $k \geq 8$, then 
Lemma \ref{ConnectedAtLeast4Poss} gives a contradiction.  Hence we must have
$k = 7$. 
Since $H'$ contains an induced copy of $\join{K_3}{2K_2}$, we must have $N(v_1) \cap
V(H) = N(v_2) \cap V(H)$, say $N(v_1) \cap V(H) = \set{w_1, w_2}$.  Moreover,
$xz \in E(G)$, $w_1w_2 \in E(G)$ and there are no edges between $\set{w_1, w_2}$
and $\set{x, z}$ in $G$.  

Put $Q \DefinedAs \set{v_1, \ldots, v_{k - (r + 1)}}$. Then for $v \in Q$, by
the same argument as above, we must have $N(v) \cap V(H) = \set{w_1, w_2}$. 
Hence $Q$ is joined to $\set{w_1, w_2}$, $C$ is joined to $Q$, and $\set{x, z}$
and both $\set{x, z}$ and $\set{w_1, w_2}$ are joined to the same $K_3$ in $H$. 
We must have $r = 3$ for otherwise one of $x, z, w_1, w_2$ has degree larger
than $7$.  Thus we have an $M_{7, 2}$ in $G$ and therefore $G$ is $M_{7,2}$, a
contradiction.

Thus we must have $\chi(D) \leq k - 1$, which gives a $(k - 1)$-coloring of $H'$ in which $z$ receives a color $c_1$ which is not received by any of the 
neighbors of $v_1$ in $H$ and $x$ receives a color $c_2$ which is not received by any of the neighbors of $v_2$ in $H$.  
Thus $c_1$ is in $v_1$'s list and $c_2$ is in $v_2$'s list. Note that if $x$ and $z$ are adjacent then $c_1 \neq c_2$. Hence, we can $2$-color $G[x,z,v_1,v_2]$ from the lists.  
This leaves $k-3$ vertices. The vertices in $C$ have lists of size at least $k-3$ and the rest have lists of size at least $k-5$.  
Since the union of any $k-4$ of the lists contains one list of size $k-3$, we
can complete the partial coloring by Hall's Theorem.
\end{proof}

\begin{cor}\label{AtMostOneEdgeIn}
For $k \geq 7$, if $H$ is a $(k - 1)$-clique in a $k$-mule $G$ other than
$M_{7,1}$ and $M_{7,2}$, then any vertex in $G - H$ has at most one neighbor in
$H$.
\end{cor}
\begin{proof}
Let $v\notin H$ be adjacent to $r$ vertices in $H$.  Now
$G[H\cup\{v\}]=\join{K_r}{(K_1+K_{k-(r+1)})}$.  If $r\ge 2$, then $G[H\cup\{v\}]$ is
forbidden by Lemma~\ref{NoForksThatArentKnives}.
\end{proof}

\begin{lem}\label{K4sOut}
For $k \geq 7$, no $k$-mule except $M_{7,1}$ contains
$\join{K_4}{E_{k-4}}$ as a subgraph.
\end{lem}
\begin{proof}
Let $G$ be a $k$-mule other than $M_{7,1}$ and suppose $G$
contains an induced copy of $\join{K_4}{D}$ where $\card{D} = k - 4$. Then $G$ is not
$M_{7,2}$. By Lemma \ref{K_tClassification}, $D$ is $E_3$, a claw, a clique, or
almost complete. If $D$ is a clique then $G$  contains $K_k$, a contradiction. Now Corollary \ref{AtMostOneEdgeIn} shows that $D$ being almost complete is
impossible. Finally, Lemma \ref{NoE3} shows that $D$ cannot be $E_3$ or a claw.  This contradiction completes the proof.
\end{proof}

Since $\join{K_4}{E_{\Delta - 4}} \subseteq K_\Delta$, Lemma \ref{K4sOut} shows that the following conjecture is equivalent to the Borodin-Kostochka conjecture.

\begin{conjecture}\label{K4Conjecture}
Any graph with $\chi = \Delta \geq 9$ contains $\join{K_4}{E_{\Delta - 4}}$ as a subgraph.
\end{conjecture}

In the next section we create the tools needed to reduce the $4$ in Lemma \ref{K4sOut} down to $3$.

\subsection{Tooling up}
For an independent set $I$ in a graph $G$, we write $\frac{G}{\brackets{I}}$ for
the graph formed by collapsing $I$ to a single vertex and discarding duplicate
edges.  We write $\brackets{I}$ for the resulting vertex in the new graph.  If
more than one independent set $I_1, I_2, \ldots, I_m$ are collapsed in
succession we indicate the resulting graph by
$\frac{G}{\brackets{I_1}\brackets{I_2}\cdots\brackets{I_m}}$.

\begin{lem}\label{ToolOne}
Let $G$ be a $k$-mule other than $M_{7,1}$ and $M_{7,2}$ with $k \geq 7$ and $H
\lhd G$. If $x, y \in V(H)$, $xy \not \in E(H)$ and  $\card{N_H(x) \cup N_H(y)} \leq k$, then there exists a $(k - 1)$-coloring $\pi$ of $H$ such that $\pi(x) = \pi(y)$.
\end{lem}
\begin{proof}
Suppose $x, y \in V(H)$, $xy \not \in E(H)$ and $\card{N_H(x) \cup N_H(y)} \leq
k$.  Put $H' \DefinedAs \frac{H}{\brackets{x, y}}$. Then
$H' \prec H$ via the natural epimorphism $\funcsurj{f}{H}{H'}$.  By applying
Lemma \ref{EpiPower} we either get the desired $(k - 1)$-coloring $\pi$ of
$H$ or a $K_{k-1}$ in $H$ with $V(K_{k-1}) \subseteq N(x) \cup N(y)$.  But $k -
1 \geq 6$, so one of $x$ or $y$ has at least three neighbors in $K_{k-1}$
violating Corollary \ref{AtMostOneEdgeIn}.
\end{proof}

\begin{lem}\label{ToolTwo}
Let $G$ be a $k$-mule other than $M_{7,1}$ and $M_{7,2}$ with $k \geq 7$ and $H
\lhd G$.  Suppose there are disjoint pairs $\set{x_1, y_1},
\set{x_2, y_2} \subseteq V(H)$ where $x_2$ and $y_2$ are nonadjacent, with $d_H(x_1)\le k-1$ and $d_H(y_1) \leq k - 1$
and $\card{N_H(x_2) \cup N_H(y_2)} \leq k$. Then there exists a $(k -
1)$-coloring $\pi$ of $H$ such that $\pi(x_1) \neq \pi(y_1)$ and $\pi(x_2) =
\pi(y_2)$.
\end{lem}
\begin{proof}
If $x_1$ is adjacent to $y_1$, we are done by Lemma \ref{ToolOne}, so assume $x_1$ and $y_1$ are nonadjacent. Put $H' \DefinedAs \frac{H}{\brackets{x_2, y_2}} + x_1y_1$.  Then
$H' \prec H$ via the natural epimorphism $\funcsurj{f}{H}{H'}$.  
Suppose the desired $(k - 1)$-coloring $\pi$ of $H$ does not exist. 
Apply Lemma \ref{EpiPower} to get a $K_k$ in $H'$. Put $z \DefinedAs
\brackets{x_2, y_2}$.  By Lemma \ref{NoE2} the $K_k$ must contain $z$ and by
Lemma \ref{NoForksThatArentKnives}, the $K_k$ must contain $x_1y_1$; hence
the $K_k$ contains $x_1$, $y_1$, and $z$.  Thus
$H$ contains an induced subgraph $A \DefinedAs \join{\set{x_1, y_1}}{K_{k-3}}$
where $V(A) \subseteq N_H(x_2) \cup N_H(y_2)$.  Suppose $x_2$ has more than two neighbors in the $K_{k-3}$.  By Lemma \ref{K4sOut}, $x_2$ has exactly three neighbors $z_1, z_2, z_3$ 
in the $K_{k-3}$.  But this is impossible by Lemma \ref{ConnectedEqual3Poss} where the $K_3$ is $G[z_1, z_2, z_3]$ and $B$ is $G\brackets{V(A) \cup \set{x_2} \setminus \set{z_1, z_2, z_3}}$.

We can make a similar argument for $y_2$, so $x_2$ and $y_2$ each have at most
two neighbors in the $K_{k-3}$. Thus $k=7$ and both $x_2$ and $y_2$ have exactly
two neighbors in the $K_4$.  One of $x_2$ or $y_2$ has at least one neighbor in
$\set{x_1, y_1}$, so by symmetry we may assume that $x_2$ is adjacent to $x_1$. 
If $x_2$ is nonadjacent to $y_1$, then $\set{x_2} \cup V(A)$ induces a $\join{K_2}{\text{antichair}}$.
Otherwise, $\set{x_2} \cup V(A) \setminus \set{z}$ induces a $\join{K_2}{C_4}$ where $z$ is one of $y_2$'s neighbors in the $K_4$.
Both possibilities are impossible by Lemma \ref{MinimalK2}.
\end{proof}

\subsection{Using our new tools}

\begin{figure}[h!tb]
\begin{center}
\includegraphics[scale=.75,trim=3cm 19.2cm 2cm 2.5cm]{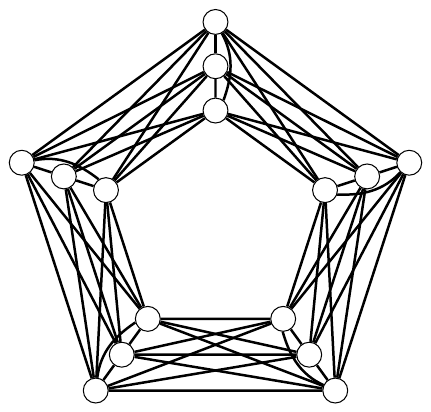}
\end{center}
\caption{The mule $M_8$, where $M_8=C_5\boxtimes K_3$.}
\label{fig:M_8}
\end{figure}

%
%

\begin{lem}\label{K3sOut}
For $k \geq 7$, the only $k$-mules containing $\join{K_3}{E_{k-3}}$
as a subgraph are $M_{7,1}$,  $M_{7,2}$ and $M_8$
(see Figures \ref{fig:M_71}, \ref{fig:M_72}, and \ref{fig:M_8}).
\end{lem}
\begin{proof}
Suppose not and let $G$ be a $k$-mule other than $M_{7,1}$,  $M_{7,2}$ and $M_8$
containing $F \DefinedAs \join{C}{B}$ as an induced subgraph where $C
= K_3$ and $B$ is an arbitrary graph with $\card{B} = k - 3$. By Lemma
\ref{ConnectedEqual3Poss}, $B$ is: $\join{E_3}{K_{\card{B} - 3}}$, almost
complete, $\djunion{K_t}{K_{\card{B} - t}}$,
$\djunion{\djunion{K_1}{K_t}}{K_{\card{B} - t - 1}}$, or
$\djunion{E_3}{K_{\card{B} - 3}}$.  The first two options are impossible by
Lemma \ref{K4sOut}.

First, suppose there is no $z \in V(G-F)$ with $C \subseteq N(z)$.  Let $\pi$
be the $(k-1)$-coloring of $G-F$ guaranteed by Lemma
\ref{JoinerOrDifferentLists}.  Put $L \DefinedAs L_\pi$. Let $I$ be a maximal
independent set in $B$. If there are $x,y \in I$ and $c \in L(x) \cap L(y)$,
then we may color $x$ and $y$ with $c$ and then greedily complete the coloring to the rest of $F$ giving a contradiction.  Thus
we must have
%
$
k - 1 \geq \sum_{v \in I} \card{L(v)} 
 \geq \sum_{v \in I} \parens{d_F(v) - 1} 
= \sum_{v \in I}(d_B(v) + 3 - 1) 
= 2 \card{I} + \sum_{v \in I} d_B(v) 
\ge \card{B} + \card{I} 
= k-3 + \card{I}. 
$

Therefore $\card{I} \leq 2$ and hence $B$ is $\djunion{K_t}{K_{\card{B} - t}}$. 
Put $N \DefinedAs \bigcup_{w \in C} N(w) \cap V(G-F)$.  Suppose $\card{N} = 1$.  
Then there must be $z \in V(C)$ with $d(z) = k-1$. Pick nonadjacent $x_2, y_2 \in V(B)$ and put
$H \DefinedAs G\brackets{V(G-F) \cup \set{x_2, y_2}}$.  Plainly, the conditions
of Lemma \ref{ToolOne} are satisfied and hence we have a $(k - 1)$-coloring
$\gamma$ of $H$ such that $\gamma(x_2) = \gamma(y_2)$.  But then we can greedily complete this coloring to all of $G$ (coloring $z$ last), a
contradiction.

Hence $\card{N} \geq 2$.  Pick $x_1,y_1 \in N$ and nonadjacent $x_2, y_2 \in V(B)$ and put
$H \DefinedAs G\brackets{V(G-F) \cup \set{x_2, y_2}}$.  Plainly, the conditions
of Lemma \ref{ToolTwo} are satisfied and hence we have a $(k - 1)$-coloring
$\gamma$ of $H$ such that $\gamma(x_1) \neq \gamma(y_1)$ and $\gamma(x_2) =
\gamma(y_2)$.  But then we can greedily complete this coloring to all of $G$, a
contradiction.

Thus we have $z \in V(G-F)$ with $C \subseteq N(z)$.  Put $B' \DefinedAs
G\brackets{V(B) \cup \set{z}}$ and $F' \DefinedAs
G\brackets{V(F) \cup \set{z}}$.  As above, using Lemma
\ref{ConnectedEqual3Poss} and Lemma \ref{K4sOut}, we see that $B'$ is
$\djunion{K_t}{K_{\card{B'} - t}}$, $\djunion{\djunion{K_1}{K_t}}{K_{\card{B'} -
t - 1}}$ or $\djunion{E_3}{K_{\card{B'} - 3}}$.  

Suppose $B'$ is $\djunion{E_3}{K_{\card{B'} - 3}}$, say the $E_3$ is
$\set{z_1, z_2, z_3}$.  Since $k \geq 7$, there exist vertices $w_1, w_2 \in
V(B') - \set{z_1, z_2, z_3}$. Then $d_{F'}(z_3) + d_{F'}(w_1) = k$ and hence we
may apply Lemma \ref{ToolOne} to get a $(k - 1)$-coloring $\zeta$ of $G - F'$
such that there is some $c \in L_\zeta(z_3) \cap L_\zeta(w_1)$.  Now
$\card{L_\zeta(z_1)} + \card{L_\zeta(z_2)} + \card{L_\zeta(w_2)} \geq 2 + 2 + k
- 4 = k$ and hence there is a color $c_1$ that is in at least two of 
$L_\zeta(z_1)$, $L_\zeta(z_2)$ and $L_\zeta(w_2)$.  If $c_1 = c$, then $c$
appears on an independent set of size $3$ in $B'$ and we may color this set
with $c$ and greedily complete the coloring. Otherwise, $B'$ contains two
disjoint nonadjacent pairs which we can color with different colors and again
complete the coloring greedily, a contradiction.

Now suppose $B'$ is $\djunion{\djunion{K_1}{K_t}}{K_{\card{B'} -
t - 1}}$.  By Lemma \ref{NoForksThatArentKnives}, we must have $2 \leq t \leq
\card{B'} - 3$. Let $x$ be the vertex in the $K_1$, $w_1, w_2 \in V(K_t)$ and
$z_1, z_2 \in V(K_{\card{B'} - t - 1})$.  Then $d_{F'}(w_1) + d_{F'}(z_1) = k
+ 1$ and hence we may apply Lemma \ref{ToolOne} to get a $(k - 1)$-coloring
$\zeta$ of $G - F'$ such that there is some $c \in L_\zeta(w_1) \cap
L_\zeta(z_1)$.  Now $\card{L_\zeta(x)} +
\card{L_\zeta(w_2)} + \card{L_\zeta(z_2)} \geq 2 + k-1 = k+1$ and hence
there are at least two colors $c_1, c_2$ that are each in at least two of 
$L_\zeta(x)$, $L_\zeta(w_2)$ and $L_\zeta(z_2)$.  If $c_1 \neq c$ or $c_2 \neq
c$, then $B'$ contains two
disjoint nonadjacent pairs which we can color with different colors and
then complete the coloring greedily.  Otherwise $c$
appears on an independent set of size $3$ in $B'$ and we may color this set
with $c$ and greedily complete the coloring, a contradiction.

Therefore $B'$ must be $\djunion{K_t}{K_{\card{B'} - t}}$.  
By Lemma \ref{NoForksThatArentKnives}, we must have $3 \leq t \leq \card{B'} -
3$.  Thus $k \geq 8$.  Let $X$ and $Y$ be the two cliques covering $B'$.  Let
$x_1, x_2 \in X$ and $y_1, y_2 \in Y$.  Put $H \DefinedAs G\brackets{V(G-F')
\cup \set{x_1, x_2, y_1, y_2}}$ and $H' \DefinedAs \frac{H}{\brackets{x_1,
y_1}\brackets{x_2,y_2}}$.  For $i \in \{1,2\}$, $d_{F'}(x_i) + d_{F'}(y_i) =
k + 2$ and thus $\Delta(H') \leq k$. If $\chi(H') \leq k - 1$, then we have a
$(k-1)$-coloring of $H$ which can be greedily completed to all of $G$, a contradiction.  
Hence, by Lemma \ref{EpiPower}, $H'$ contains $K_k$.  Thence $H - \set{x_1,
y_1, x_2, y_2}$ contains $K_{k-2}$, call it $A$, such that $V(A) \subseteq
N(x_i) \cup N(y_i)$ for $i \in \{1,2\}$.  
Since $d_{F'}(x_i)+d_{F'}(y_i)=k+2$, we see that
$N_H(x_i) \cap N_H(y_i) = \emptyset$ for $i \in \{1,2\}$.  But we can play
the same game with the pairs $\set{x_1, y_2}$ and $\set{x_2, y_1}$.  We conclude
that $N(x_1) \cap V(A) = N(x_2) \cap V(A)$ and $N(y_1) \cap V(A) = N(y_2) \cap
V(A)$.  In fact we can extend this equality to all of $X$ and $Y$.  

Put $Q \DefinedAs N(x_1) \cap V(A)$ and $P \DefinedAs N(y_1) \cap V(A)$.  Then
we conclude that $X$ is joined to $Q$ and $Y$ is joined to $P$.  Moreover, we
already know that $X$ and $Y$ are joined to the same $K_3$.  The edges in these
joins exhaust the degrees of all the vertices, hence $G$ is a $5$-cycle with
vertices blown up to cliques.  If $k = 8$, then $\card{X} = \card{Y} = 3$ and
thus $\card{Q} = \card{P} = 3$, but then $G = M_8$, a contradiction.  So $k
\geq 9$.  Since $\card{X}+\card{Y}= k-2 \ge 7$, we have either $\card{X}\ge 4$
or $\card{Y}\ge 4$.
If $\card{X}\ge 4$, then for each $q\in Q$, we have $d(q)\ge
(k-2)-1+\card{X}\ge k+1$, contradiction.  If $\card{Y}\ge 4$, then for each
$p\in P$, we have $d(p)\ge (k-2)-1+\card{Y}\ge k+1$, contradiction.
\end{proof}

Since $\join{K_3}{E_{\Delta-3}} \subseteq K_\Delta$, Lemma \ref{K3sOut} shows
that Conjecture \ref{K3Conjecture} is equivalent to the Borodin-Kostochka
conjecture.  

\begin{cor}
\label{MainResult}
The Borodin-Kostochka conjecture is equivalent to the following:
Any graph with $\chi = \Delta \geq 9$ contains $\join{K_3}{E_{\Delta-3}}$ as a
subgraph.
\end{cor}

In the extended version of this paper \cite{mules_big}, a more detailed analysis of $d_1$-choosable joins was used to prove the following.

\begin{lem}\label{NonInducedThreeDealInMule}
Let $G$ be a $k$-mule with $k \geq 7$ other than $M_{7,1}$, $M_{7,2}$ and $M_8$. 
Let $A$ and $B$ be graphs with $3 \leq \card{A} \leq k - 3$ and $\card{B} = k - \card{A}$ such that $\join{A}{B} \unlhd G$. 
Then $A = \djunion{K_1}{K_{\card{A} - 1}}$ and $B = \djunion{K_1}{K_{\card{B} - 1}}$.
\end{lem}

Lemma~\ref{NonInducedThreeDealInMule} implies that the following (seemingly weaker) conjecture is equivalent to the Borodin-Kostochka conjecture. 

\begin{conjecture}\label{NoThreeDealEquiv}
Any graph with $\chi = \Delta \geq 9$ contains some $\join{A_1}{A_2}$ as an
induced subgraph where $\card{A_1}, \card{A_2} \geq 3$, $\card{A_1} + \card{A_2} = \Delta$
and $A_i \neq \djunion{K_1}{K_{\card{A_i} - 1}}$ for some $i \in \{1,2\}$.
\end{conjecture}

The condition $A_i \neq \djunion{K_1}{K_{\card{A_i} - 1}}$ is unnatural and
by removing it we get a (possibly) weaker conjecture than the
Borodin-Kostochka conjecture which has more aesthetic appeal.

\begin{conjecture}\label{NonInducedThreeDeal}
Let $G$ be a graph with $\Delta(G) = k \geq 9$. If $K_{t, k - t} \not \subseteq G$ for all $3 \leq t \leq k - 3$, then $G$ can be $(k - 1)$-colored.
\end{conjecture}

\begin{conjecture}
Conjecture \ref{NonInducedThreeDeal} is equivalent to the Borodin-Kostochka
conjecture.
\end{conjecture}

Perhaps it would be easier to attack Conjecture \ref{NonInducedThreeDeal} with
the condition $3 \leq t \leq k - 3$ in the hypothesis
replaced by $2 \leq t \leq k - 2$.  We are unable to
prove even this weakened version of Conjecture~\ref{NonInducedThreeDeal}. 
Making this change and bringing $k$ down to $5$
gives Conjecture \ref{NonInducedTwoDeal}, which, if true, would imply the
remaining two cases of Gr\"unbaum's girth problem for graphs with girth at
least 5.  Specifically, Gr\"unbaum conjectured that for every $k\ge 3$ and
$g\ge 4$ there exists a $k$-chromatic $k$-regular graph with no cycle of length
less than $g$.  For most pairs $(k,g)$ the problem has now been resolved.  For
further discussion and more references, see Problem 7.2
of~\cite{GraphColoringProblems}.

\begin{conjecture}\label{NonInducedTwoDeal}
Let $G$ be a graph with $\Delta(G) = k \geq 5$. If $K_{t, k - t} \not \subseteq
G$ for all $2 \leq t \leq k - 2$, then $G$ can be $(k - 1)$-colored.
\end{conjecture}

If $G$ is a graph with with $\Delta(G) = k \geq 5$ and girth at least 5, then
it contains no $K_{t, k - t}$ for all $2 \leq t \leq k - 2$ and hence
Conjecture \ref{NonInducedTwoDeal} would give a
$(k-1)$-coloring.  This conjecture would be tight since the Gr\"unbaum graph and
the Brinkmann graph are examples with $\chi = \Delta = 4$ and girth at least
5.

%% file: M_61.tex
\begin{figure}[htb]
\centering
\begin{tikzpicture}[scale = 14]
\tikzstyle{VertexStyle}=[shape = circle,	
								 minimum size = 1pt,
								 inner sep = 3pt,
                         draw]
\Vertex[x = 0.25085711479187, y = 0.92838092893362, L = \tiny {}]{v0}
\Vertex[x = 0.0932380929589272, y = 0.929142817854881, L = \tiny {}]{v1}
\Vertex[x = 0.250571489334106, y = 0.781142801046371, L = \tiny {}]{v2}
\Vertex[x = 0.092571459710598, y = 0.781142845749855, L = \tiny {}]{v3}
\Vertex[x = 0.355238050222397, y = 0.889142841100693, L = \tiny {}]{v4}
\Vertex[x = 0.353904783725739, y = 0.853142827749252, L = \tiny {}]{v5}
\Vertex[x = 0.353238135576248, y = 0.818476170301437, L = \tiny {}]{v6}
\Vertex[x = 0.476000010967255, y = 0.968571435660124, L = \tiny {}]{v7}
\Vertex[x = 0.537999987602234, y = 0.853238105773926, L = \tiny {}]{v8}
\Vertex[x = 0.444666564464569, y = 0.77590474486351, L = \tiny {}]{v9}
\Vertex[x = 0.475333333015442, y = 0.731904745101929, L = \tiny {}]{v10}
\Vertex[x = 0.451999962329865, y = 0.916571423411369, L = \tiny {}]{v11}
\Edge[](v0)(v1)
\Edge[](v2)(v1)
\Edge[](v3)(v1)
\Edge[](v0)(v3)
\Edge[](v2)(v3)
\Edge[](v2)(v0)
\Edge[](v4)(v2)
\Edge[](v5)(v2)
\Edge[](v6)(v2)
\Edge[](v4)(v0)
\Edge[](v5)(v0)
\Edge[](v6)(v0)
\Edge[](v4)(v1)
\Edge[](v5)(v1)
\Edge[](v6)(v1)
\Edge[](v4)(v3)
\Edge[](v5)(v3)
\Edge[](v6)(v3)
\Edge[](v8)(v7)
\Edge[](v9)(v7)
\Edge[](v10)(v7)
\Edge[](v11)(v7)
\Edge[](v9)(v8)
\Edge[](v10)(v8)
\Edge[](v11)(v8)
\Edge[](v10)(v9)
\Edge[](v11)(v9)
\Edge[](v11)(v10)
\Edge[](v4)(v7)
\Edge[](v4)(v11)
\Edge[](v6)(v9)
\Edge[](v6)(v10)
\Edge[](v5)(v8)
\Edge[](v5)(v9)
\end{tikzpicture}
\caption{The mule $M_{6,1}$.}
\label{fig:M_61}
\end{figure}

%% file: M7.tex
\begin{figure}[htb]
\centering
\begin{tikzpicture}[scale = 10]
\tikzstyle{VertexStyle}=[shape = circle,	
								 minimum size = 1pt,
								 inner sep = 3pt,
                         draw]
\Vertex[x = 0.751999914646149, y = 0.724000096321106, L = \tiny {}]{v0}
\Vertex[x = 0.751999974250793, y = 0.590000092983246, L = \tiny {}]{v1}
\Vertex[x = 0.652000069618225, y = 0.38400000333786, L = \tiny {}]{v2}
\Vertex[x = 0.578000009059906, y = 0.51800012588501, L = \tiny {}]{v3}
\Vertex[x = 0.572000086307526, y = 0.808000028133392, L = \tiny {}]{v4}
\Vertex[x = 0.0419999808073044, y = 0.742000013589859, L = \tiny {}]{v5}
\Vertex[x = 0.0399999916553497, y = 0.612000048160553, L = \tiny {}]{v6}
\Vertex[x = 0.163999989628792, y = 0.569999992847443, L = \tiny {}]{v7}
\Vertex[x = 0.25600004196167, y = 0.676000028848648, L = \tiny {}]{v8}
\Vertex[x = 0.159999996423721, y = 0.782000005245209, L = \tiny {}]{v9}
\Vertex[x = 0.653999924659729, y = 0.921999998390675, L = \tiny {}]{v10}
\Vertex[x = 0.379999995231628, y = 0.771999999880791, L = \tiny {}]{v11}
\Vertex[x = 0.381999999284744, y = 0.682000011205673, L = \tiny {}]{v12}
\Vertex[x = 0.386000007390976, y = 0.592000007629395, L = \tiny {}]{v13}
\Edge[](v0)(v4)
\Edge[](v1)(v4)
\Edge[](v2)(v4)
\Edge[](v3)(v4)
\Edge[](v0)(v3)
\Edge[](v1)(v3)
\Edge[](v2)(v3)
\Edge[](v0)(v2)
\Edge[](v1)(v2)
\Edge[](v0)(v1)
\Edge[](v5)(v6)
\Edge[](v5)(v7)
\Edge[](v5)(v8)
\Edge[](v5)(v9)
\Edge[](v6)(v7)
\Edge[](v6)(v8)
\Edge[](v6)(v9)
\Edge[](v7)(v8)
\Edge[](v7)(v9)
\Edge[](v8)(v9)
\Edge[](v0)(v10)
\Edge[](v1)(v10)
\Edge[](v2)(v10)
\Edge[](v3)(v10)
\Edge[](v4)(v10)
\Edge[](v5)(v11)
\Edge[](v6)(v11)
\Edge[](v7)(v11)
\Edge[](v8)(v11)
\Edge[](v9)(v11)
\Edge[](v5)(v12)
\Edge[](v6)(v12)
\Edge[](v7)(v12)
\Edge[](v8)(v12)
\Edge[](v9)(v12)
\Edge[](v5)(v13)
\Edge[](v6)(v13)
\Edge[](v7)(v13)
\Edge[](v8)(v13)
\Edge[](v9)(v13)
\Edge[](v11)(v10)
\Edge[](v11)(v4)
\Edge[](v12)(v0)
\Edge[](v12)(v1)
\Edge[](v13)(v3)
\Edge[](v13)(v2)
\end{tikzpicture}
\caption{The mule $M_{7,1}$.}
\label{fig:M_7}
\label{fig:M_71}
\end{figure}

%% file: ListColoringLemmas.tex
\section{List coloring lemmas}
\label{ListColoringSection}
In this section we use list-coloring lemmas to forbid a large class of graphs
from appearing as induced subgraphs of mules.  In each case, we assume that
such a graph
$H\lhd G$ appears as an induced subgraph of a mule $G$.  By the minimality of
$G$, we can color $G - H$ with $\Delta-1$ colors.  If $H$ can be colored
regardless of which colors are forbidden by its colored neighbors in $G -
H$, then we can clearly extend this coloring to all of $G$.  

Let $G$ be a graph.  A \emph{list assignment} to the vertices of $G$ is a
function from $V(G)$ to the finite subsets of $\mathbb{N}$.  A list assignment
$L$ to $G$ is \emph{good} if $G$ has a coloring $c$ where $c(v) \in L(v)$ for
each $v \in V(G)$.  It is \emph{bad} otherwise.  We call the collection of all
colors that appear in $L$, the \emph{pot} of $L$.  That is $Pot(L) \DefinedAs
\bigcup_{v \in V(G)} L(v)$.  
For a subgraph $H$ of $G$ we write $Pot_H(L)
\DefinedAs \bigcup_{v \in V(H)} L(v)$. For $S \subseteq Pot(L)$, let $G_S$ be
the graph $G\left[\setb{v}{V(G)}{L(v) \cap S \neq \emptyset}\right]$.  We also
write $G_c$ for $G_{\{c\}}$. 
For $\func{f}{V(G)}{\IN}$, an $f$-assignment on $G$ is an
assignment $L$ of lists to the vertices of $G$ such that $\card{L(v)} = f(v)$
for each $v \in V(G)$.  We say that $G$ is {\it $f$-choosable} if every
$f$-assignment on $G$ is good.

\subsection{Shrinking the pot}
In this section we prove a lemma about bad list assignments with minimum pot size.
Some form of this lemma has appeared independently in at least two places we
know of: Kierstead \cite{kierstead2000choosability} and Reed and Sudakov \cite{ReedSudakov}.  
We will use the following lemma frequently throughout
the remainder of this paper.  

Our approach to coloring a graph (particularly a join) will often be to
consider nonadjacent vertices $u$ and $v$ and show that their lists contain a
common color.  By the pigeonhole principle, this follows immediately when
$|L(u)|+|L(v)|>|Pot(L)|$.  Thus, it is convenient to bound the size of
$|Pot(L)|$.  Intuitively, the lemma says that when we are trying to
color a graph $G$ from a list assignment $L$, we may assume that $|Pot(L)|<|G|$.

\begin{SmallPotLemma}
Let $f: V(G) \to \{1,\ldots,|G|-1\}$.  If $G$ is not $f$-choosable, then there
is a bad list assignment $L$ such that $|L(v)|=f(v)$ for all $v\in V(G)$ and
$|Pot(L)|<|G|$.
\end{SmallPotLemma}

\begin{proof}
Fix $G$ and $f$.  Suppose to the contrary that $G$ is not $f$-choosable, but
that $G$ has an $L$-coloring whenever $|L(v)|=f(v)$ for all $v$ and
$|Pot(L)|<|G|$.
For any $U \subseteq V$ and any list assignment $L$, let
$L(U)$ denote $\bigcup_{v\in U}L(v)$.  Let $L$ be an $f$-assignment such that
$|Pot(L)| \ge |G|$ and $G$ is not $L$-colorable.  For each $U \subseteq V$, let
$g(U) = |U| - |L(U)|$.  Since $G$ is not $L$-colorable, Hall's Theorem implies
there exists $U$ with $g(U)>0$.  Choose $U$ to maximize $g(U)$.  Let $A$ be an
arbitrary set of $|G|-1$ colors containing $L(U)$.  Construct $L'$ as follows. 
For $v\in U$, let $L'(v)=L(v)$.  Otherwise, let $L'(v)$ be an arbitrary subset
of $A$ of size $f(v)$.  Now $|Pot(L')| < |G|$, so by hypothesis, $G$ has an
$L'$-coloring.  This gives an $L$-coloring of $U$.  By the maximality of
$g(U)$, for $W \subseteq (V\setminus U)$, we have $|L(W)\setminus L(U)| \ge
|W|$.  Thus, by Hall's Theorem, we can extend the $L$-coloring of $U$ to all of
$V$; this contradicts the fact that $G$ is not $L$-colorable, and hence finishes
the proof.
\end{proof}

\subsection{Degree choosability}
\label{Degree choosability}
\begin{defn}
Let $G$ be a graph and $r \in \mathbb{Z}$.  Then $G$ is \emph{$d_r$-choosable}
if $G$ is $f$-choosable where $f(v) = d(v) - r$.
\end{defn}

In the extended version of this paper \cite{mules_big}, we characterize all
graphs $\join{A}{B}$ with $|A|\ge 2$, $|B|\ge 2$ that are not $d_1$-choosable. 
Since the proof of this characterization is lengthy, here we only prove what is
necessary for our main result.

Note that a vertex critical graph with $\chi = \Delta + 1 - r$ contains no
induced $d_r$-choosable subgraph.  Since we are working to prove the
Borodin--Kostochka conjecture, we will focus on the case $r=1$ and primarily
study $d_1$-choosable graphs.  For $r=0$, we have the following well known generalization of Brooks' Theorem (see \cite{borodin1977criterion}, \cite{erdos1979choosability}, \cite{BrooksExtended}, \cite{Entringer1985367} and \cite{Hladky}).
A \emph{Gallai tree} is a graph all of whose blocks are complete graphs or odd cycles.

\begin{ClassificationOfd0}
For any connected graph $G$, the following are equivalent.
\begin{itemize}
\item $G$ is $d_0$-choosable.
\item $G$ is not a Gallai tree.
\item $G$ contains an induced even cycle with at most one chord.
\end{itemize}
\end{ClassificationOfd0}

Throughout Section~\ref{ListColoringSection}, we will often extend a partial
coloring by \emph{coloring greedily}.  By this we mean that we consider the
vertices in some order and color each vertex with a color not already used on any of
its neighbors.  In
particular, if $L$ is a $d_0$-list assignment and each vertex in our order
(except the last) has some neighbor later in the order, then we can always
greedily color all vertices but the last.  When we \emph{color $V(H_1)$
greedily toward $H_2$} (for some connected subgraphs $H_1$ and $H_2$ with
$H_2\lhd H_1$), we order the vertices of $H_1\setminus H_2$ by nonincreasing
distance from $H_2$.  At the time we consider each vertex $v$, it will have an
uncolored neighbor on a shortest path to $H_2$, so we will have a color for $v$.
Furthermore, if $L$ is a list assignment with $|L(v)|\ge d(v)$ for all $v$ and
$|L(w)|>d(w)$ for some $w$, then we can color greedily toward $w$ and color $w$
last.

We give a couple lemmas about $d_0$-assignments that will be useful in our study
of $d_1$-assignments.  The following lemma was used in \cite{BrooksExtended}.

\begin{lem}\label{d0NeighborList}
Let $L$ be a bad $d_0$-assignment on a connected graph $G$ and $x \in V(G)$ a noncutvertex. Then $L(x) \subseteq L(y)$ for each $y \in N(x)$.
\end{lem}
\begin{proof}
Suppose otherwise that we have $c \in L(x) - L(y)$ for some $y \in N(x)$.
Coloring $x$ with $c$ leaves at worst a $d_0$-assignment $L'$ on the connected
$H \DefinedAs G-x$ where $\card{L'(y)} > d_H(y)$.  But then we can complete the
coloring, by coloring greedily toward $y$.
\end{proof}

The following lemma is similar to a special case of the Small Pot Lemma, but the
key difference is that here we need not assume that we have a bad
$d_0$-assignment with minimum pot size.

\begin{lem}\label{d0PotColoring}
If $L$ is a bad $d_0$-assignment on a connected graph $G$, $\card{Pot(L)} < \card{G}$.
\end{lem}
\begin{proof}
Suppose that the lemma is false and choose a connected graph $G$ together with a
bad $d_0$-assignment $L$ where $\card{Pot(L)} \geq \card{G}$ minimizing
$\card{G}$.  Plainly, $\card{G} \geq 2$.  Let $x \in G$ be a noncutvertex (any
end block has at least one).  By Lemma \ref{d0NeighborList}, $L(x) \subseteq
L(y)$ for each $y \in N(x)$. Thus coloring $x$ yields a bad $d_0$-assignment of
$G-x$ with the pot decreased by at most one, giving a smaller counterexample.  This contradiction completes the proof.
\end{proof}

Now we are able to characterize graphs $A$ such that $\join{K_t}{A}$ is
$d_1$-choosable, when $t\ge 4$.

\begin{defn}
A graph $G$ is \emph{almost complete} if $\omega(G) \geq \card{G} - 1$.
\end{defn}

\begin{lem}\label{ConnectedAtLeast4Poss}
\label{K_tClassification}
For $t \geq 4$, $\join{K_t}{B}$ is $d_1$-choosable unless 
$B$ is almost complete;
or $t = 4$ and $B$ is $E_3$ or $K_{1,3}$; or $t = 5$ and $B$ is $E_3$.
\end{lem}
\begin{proof}
Suppose the lemma is false and let $t$ and $B$ form a counterexample.
Since $B$ is not almost complete, $B$ contains either an independent set $\set{x_1, x_2, x_3}$ or two disjoint pairs of nonadjacent vertices
$\set{x_1, x_2}$ and $\set{x_3, x_4}$.  If $B$ does not contain two disjoint
pairs of nonadjacent vertices, then (by possibly moving dominating vertices
from $B$ to $K_t$) we have $B =  E_3$ and $t \ge 6$.  Thus it will suffice to
prove that $\join{K_6}{E_3}$ is $d_1$-choosable and that $\join{K_4}{H}$ is
$d_1$-choosable where $H$ is a spanning subgraph of $C_4$.  When the graph is
larger, we can greedily color all other vertices, since each will have at least
two uncolored neighbors in the clique.

Suppose $G \DefinedAs \join{K_6}{E_3}$ is not $d_1$-choosable and let $L$ be a
bad $d_1$-assignment on $G$ with $|Pot(L)|$ minimum.  By the Small Pot Lemma, $\card{Pot(L)}
\le 8$.  For every distinct pair $i,j\in\{1,2,3\}$, we have $|L(x_i)| +
|L(x_j)| = 10 > \card{Pot(L)}$, so each pair $x_i$ and $x_j$ have a common
color.  Suppose there is some vertex $y\in V(K_6)$ and $i\in\{1,2,3\}$ and
color $c$ such that $c\in L(x_i)\setminus L(y)$.  Now we can use $c$ on $x_i$,
use a common color on the two remaining $x_j$'s, then finish greedily, ending
with $y$.  Thus, we have $L(x_i)\subseteq L(y)$ for all $i$ and all $y$, which
implies that $\card{Pot(L)}=7$.
But then $|L(x_1)| + |L(x_2)|
+ |L(x_3)| = 15 > 2\card{Pot(L)}$, so by the pigeonhole principle $x_1, x_2,
x_3$ share a common color $c$.  We use $c$ on $x_1,x_2,x_3$ and complete the
coloring greedily to all of $G$, a contradiction.

Suppose $G \DefinedAs \join{K_4}{H}$ is not $d_1$-choosable for some $H$
a spanning subgraph of $C_4$ and let $L$ be a bad $d_1$-assignment on $G$ with
$|Pot(L)|$ minimum.  By the Small Pot Lemma, $\card{Pot(L)} \le 7$. Let $y$ be
an arbitrary vertex in $V(K_4)$.  If $H$ contains at least two edges, then
$|L(x_1)|+|L(x_2)|\ge 8$ and $|L(x_3)|+|L(x_4)|\ge 8$.  So $L(x_1)\cap
L(x_2)\ne \emptyset$.
Color $x_1$ and $x_2$ with a common color $c_1$.  If $|L(y)-c_1|= 5$, then we
can color $x_3$ with some color $c_2$ and $x_4$ with some color $c_3$ so that
$|L(y)-c_1-c_2-c_3|=4$ (since $|L(x_3)-c_1|+|L(x_4)-c_1|>|L(y)-c_1|$); otherwise
$\card{L(y)-c_1}=6$.  In each case, we now finish greedily, ending with $y$.

Suppose $H$ contains exactly one edge $x_1x_3$.  Similar to the previous
argument, $L(x_1)\cap L(x_2)=\emptyset$; otherwise, use a common color on $x_1$
and $x_2$, possibly save a color on $y$ via $x_3$ and $x_4$, then
finish greedily.  By symmetry, $L(x_1)=L(x_3)=\{a,b,c,d\}$ and
$L(x_2)=L(x_4)=\{e,f,g\}$.  Also by symmetry, $a$ or $e$ is missing from
$L(y)$.  So color $x_1$ with $a$ and $x_2$ and $x_4$ with $e$ and $x_3$
arbitrarily; now finish greedily, ending with $y$.  So instead $H=E_4$.  If a
common color appears on 3 vertices of $S$, use it there, then finish greedily.
If not, then by pigeonhole, at least 5 colors appear on pairs of vertices.  
Color two disjoint pairs, each with a common color.  Now finish the coloring
greedily.
\end{proof}

For characterizing graphs $\join{K_3}{A}$ which are $d_1$-choosable, we would
like to take a similar approach to what we did in the previous lemma, but it
does not quite work.  When $t \ge 4$, each vertex $v$ has $|L(v)|\ge d_A(v)+3$.  
When $t=3$, we only have $|L(v)|\ge d_A(v)+2$.  As a result, we are not
guaranteed as much overlap between lists of nonadjacent vertices in $A$.  To
overcome this problem and guarantee a larger overlap among these lists, we need
to improve our upper bound on $\card{Pot(L)}$.  In the next section we do this.

\subsection{Shrinking the pot further}
The Small Pot Lemma implies that if $\join{A}{B}$ is not $d_1$-choosable, then
$\join{A}{B}$ has a bad $d_1$-assignment $L$ such that $|Pot(L)|\le |A|+|B|-1$.
In this section, we study conditions under which we can assume $|Pot(L)|\le |A|+|B|-2$.  
In the previous section, we characterized graphs $\join{K_t}{B}$ which are
$d_1$-choosable when $t\ge 4$ (in fact, the same proof shows $d_1$-choosability
if the $K_t$ is replaced by any connected graph with at least 4 vertices).
Thus, for the present section, the reader
should keep in mind the case $|A|=3$.
In the following section, our results here help us to find nonadjacent
vertices with a common color, and ultimately to characterize graphs
$\join{K_3}{B}$ that are $d_1$-choosable.

\begin{lem}\label{ConnectedPot}
Let $A$ be a connected graph and $B$ an arbitrary graph such that $\join{A}{B}$
is not $d_1$-choosable.  Let $L$ be a bad $d_1$-assignment on
$\join{A}{B}$ with $|Pot(L)|$ minimum.  If $B$ is colorable from $L$ using at most $\card{B} - 1$
colors, then $\card{Pot(L)} \leq \card{A} + \card{B} - 2$.
\end{lem}
\begin{proof}
To get a contradiction suppose that $\card{Pot(L)} \geq \card{A} + \card{B} - 1$
and that $B$ is colorable from $L$ using at most $\card{B} - 1$ colors.  If
$\card{Pot_A(L)} = \card{Pot(L)}$, then coloring $B$ with at most
$|B|-1$ colors leaves at worst a $d_0$-assignment $L'$ on $A$ with
$\card{Pot(L')} \ge \card{A}$ (here $d_0$ refers to the degrees in $A$).  
Hence the coloring can be completed to $A$ by Lemma \ref{d0PotColoring}, a
contradiction.

Thus we may assume that $\card{Pot_A(L)} \le \card{Pot(L)} - 1$. Put
$S\DefinedAs Pot(L) - Pot_A(L)$.  Since $S$ is nonempty, choose an arbitrary
color $c\in S$.  Let $\pi$ be a coloring of $B$ from $L$ using at most
$\card{B} - 1$ colors, and color $B$ with $\pi$.  Now $\pi$ uses $|B|-1$ colors
on $B$, and none of these colors are in $S$, for otherwise $A$ has at worst a
$d_{-1}$-assignment.
In other words, all vertices of $B$ are colored
with distinct colors, except for one nonadjacent pair $x,y$.  

If we can change the color of any vertex of $B-x-y$ to some color in $S$, then 
again $A$ has at worst a $d_{-1}$-assignment.  
Thus, by symmetry (between $x$ and $y$),
$S \subseteq L(x)$ and $Pot_{B-x}(L) \cap S = \emptyset$; in particular, $c\in
L(x)\setminus Pot_{B-x}(L)$.  
Since $d(x) \le |A| + |B| - 1$, we have $\card{L(x)} \le |A| + |B| - 2$, so
there is $c' \in Pot(L) - L(x)$.  Consider the list assignment $L'$ on $G$
defined by $L'(v) = L(v)$ for $v \ne x$ and $L'(x) = L(x) \cup \set{c'} -
\set{c}$.  By minimality of $L$, $G$ has an $L'$-coloring from which we get an
$L$-coloring by using $c$ on $x$, a contradiction.
\end{proof}

To apply Lemma~\ref{ConnectedPot}, in the next lemma we give a condition under
which $B$ can be colored with at most $|B|-1$ colors.  In the lemma after that,
we show that the condition holds whenever $B$ is not the disjoint union of at
most two complete subgraphs.  

\begin{lem}\label{MaxindependentCondition}
Let $A$ be a graph with $\card{A} \geq 2$, $B$ an arbitrary graph and $L$ a
$d_1$-assignment on $\join{A}{B}$.  If $B$ has an independent set $I$ such that
$(\card{A} - 1)\card{I} + \sum_{v\in I}d_B(v) > \card{Pot(L)}$, then $B$ can be colored from $L$ using at most $\card{B} - 1$ colors.
\end{lem}
\begin{proof}
Suppose that $B$ has an independent set $I$ such that $(\card{A} - 1)\card{I} +
\sum_{v\in I}d_B(v) > \card{Pot(L)}$.  Now

$$
\sum_{v\in I}|L(v)| = \sum_{v\in I}(d(v)-1) 
= (|A|-1)|I| + \sum_{v\in I}d_B(v) = (|A|-1)|I| + \sum_{v\in I}d_B(v) > |Pot(L)|.
$$

Hence we have distinct $x, y \in I$ with a common color $c$ in their lists.  So
we color $x$ and $y$ with $c$.  Since $\card{A} \geq 2$, this leaves at worst a $d_{-1}$-assignment on the rest of $B$.  Completing the coloring to the rest of $B$ gives the desired coloring of $B$ from $L$ using at most $\card{B} - 1$ colors.
\end{proof}

\begin{lem}\label{BasicZeta}
Let $G$ be a graph and $I$ a maximal independent set in $G$. Then $\sum_{v\in
I}d(v) \geq \card{G} - \card{I}$.  If $I$ is maximum and $\sum_{v\in I}d(v) = \card{G} - \card{I}$, then $G$ is the disjoint union of $\card{I}$ complete graphs.
\end{lem}
\begin{proof}
Each vertex in $G-I$ is adjacent to at least one vertex in $I$.  Hence
$\sum_{v\in I}d(v) \geq \card{G} - \card{I}$.

Now assume $I$ is maximum and $\sum_{v\in I}d(v) = \card{G} - \card{I}$.  Then $N(x)
\cap N(y) = \emptyset$ for every distinct pair $x,y \in I$.  Also, $N(x)$ must
be a clique for each $x \in I$, since otherwise we could swap $x$ out for a pair of nonadjacent neighbors and get a larger independent set.  Since we can swap $x$ with any of its neighbors to get another maximum independent set, we see that $G$ has components $\setbs{G[\set{v} \cup N(v)]}{v \in I}$.
\end{proof}

Now we combine Lemmas~\ref{ConnectedPot}--\ref{BasicZeta}.

\begin{lem}\label{ConnectedEqual3}
Let $A$ be a connected graph with $\card{A} = 3$ and $B$ a graph that is not the disjoint union of at most two complete subgraphs.  If $\join{A}{B}$ is not $d_1$-choosable, then $\join{A}{B}$ has a bad $d_1$-assignment $L$ such that $\card{Pot(L)} \leq \card{B} + 1$.
\end{lem}
\begin{proof}
Suppose $\join{A}{B}$ is not $d_1$-choosable and let $L$ be a bad $d_1$-assignment on $\join{A}{B}$ with $|Pot(L)|$ minimum.  Then, by the Small Pot Lemma, $\card{Pot(L)} \leq \card{B} + 2$.  

Let $I$ be a maximum independent set in $B$. Since $B$ is not the disjoint union of
at most two complete subgraphs, Lemma~\ref{BasicZeta} implies that either
$\sum_{v\in I}d(v) > \card{B} - \card{I}$ or $\card{I} \geq 3$.  In the first
case, $2\card{I} + \sum_{v\in I}d(v) > 2\card{I} + \card{B} - \card{I} \geq 2 + \card{B} \geq \card{Pot(L)}$.  In the second case, 
$2\card{I} + \sum_{v\in I}d(v) \geq 2\card{I} + \card{B} - \card{I} \geq 3 + \card{B} > \card{Pot(L)}$.

Thus by Lemma \ref{MaxindependentCondition}, $B$ can be colored from $L$ using at most $\card{B} - 1$ colors.  But then we are done by Lemma \ref{ConnectedPot}.
\end{proof}

\subsection{Joins with $K_3$}
\label{JoinsWithThree}
In this section we investigate the $d_1$-choosable graphs of the form
$\join{K_3}{B}$, where $\card{B} \geq 2$.

\begin{lem}\label{AJoinP_4}
$\join{K_3}{P_4}$ is $d_1$-choosable.
\end{lem}
\begin{proof}
Suppose otherwise, and let $G$ be $\join{K_3}{P_4}$ and let $L$ be a bad
list assignment with $|Pot(L)|$ minimum.  Now Lemma~\ref{ConnectedEqual3} implies that $|Pot(L)|\le |P_4|+1=5$.
Denote the vertices of the $P_4$ as $y_1, y_2, y_3, y_4$ in order.  Note that
$|L(y_1)|+|L(y_3)|=3+4> \card{Pot(L)}+1$ and $|L(y_2)|+|L(y_4)|=4+3>
\card{Pot(L)}+1$.  
Hence we have colors $c_1 \neq c_2$ with $c_1 \in L(y_1) \cap L(y_3)$ and $c_2
\in L(y_2) \cap L(y_4)$.  We color the vertices in each pair with their
respective colors and finishing greedily on the $K_3$.
\end{proof}

\input{antipaw}

\begin{lem}\label{AntiPaw}
$\join{K_3}{\text{antipaw}}$ is $d_1$-choosable.
\end{lem}
\begin{proof}
Suppose not. We use the labeling of the antipaw given in
Figure \ref{fig:antipaw}. Since the antipaw is not a disjoint union of at most
two complete graphs, Lemma \ref{ConnectedEqual3} gives us a bad
$d_1$-assignment $L$ on $\join{K_3}{\text{antipaw}}$ with $\card{Pot(L)} \leq
5$.  Note that $\card{L(y_1)} + \card{L(y_4)} \geq 6$ and $\card{L(y_2)} +
\card{L(y_3)} \geq 6$.  We must have $\card{L(y_1)
\cap L(y_4)} = 1$ and $L(y_1) \cap L(y_4) = L(y_2) \cap L(y_3)$, for otherwise we color each pair with a different color and finish greedily.  But then we
have $c \in L(y_2) \cap L(y_3) \cap L(y_4)$ and after coloring $y_2, y_3, y_4$
with $c$ we can complete the coloring, getting a contradiction.
\end{proof}

\begin{lem}\label{ConnectedEqual3Poss}
\label{K3Classification}
$\join{K_3}{B}$ is $d_1$-choosable unless
\begin{enumerate}
\item[(a)]
$B$ is almost complete, 
\item[(b)]
$B\in\{\djunion{K_t}{K_{\card{B} - t}}, \djunion{\djunion{K_1}{K_t}}{K_{\card{B} - t - 1}}, \djunion{E_3}{K_{\card{B}
- 3}}\}$, or 
\item[(c)]
$\card{B} \leq 5$ and $B = \join{E_3}{K_{\card{B} - 3}}$.
\end{enumerate}
\end{lem}
\begin{proof}
Let $\join{K_3}{B}$ be a graph that is not $d_1$-choosable and let $B$ be none
of the specified graphs.  Lemma \ref{ConnectedEqual3} gives us a bad
$d_1$-assignment $L$ on $\join{K_3}{B}$ with $\card{Pot(L)} \leq \card{B} + 1$.
Furthermore, the proof of Lemma~\ref{ConnectedEqual3} shows that we can color
$B$ with at most $|B|-1$ colors.  
In particular we have nonadjacent $x, y \in V(B)$ and $c \in L(x) \cap L(y)$.  
Coloring $x$ and $y$ with $c$, and removing $c$ from the lists of their
neighbors, leaves a list assignment $L'$ on $D \DefinedAs B
- \set{x, y}$.  
If $c \in L'(z)$ for some $z \in V(D)$, then $\set{x, y, z}$ is independent and
we can color $z$ with $c$ and complete the coloring to get a contradiction.  
Hence $Pot(L') \subseteq Pot(L) - \set{c}$, and thus $\card{Pot(L')}\le |B|$.  
If we can find any nonadjacent pair in $D$ to receive a common color, then
we can finish the coloring greedily.  Now the fact that all nonadjacent pairs
in $D$ must have disjoint lists greatly restricts the possibilities for $D$.

Suppose, for a contradiction, that $D$ is not the disjoint union of at most two
complete subgraphs.  If $\alpha(D) \geq 3$, let $J$ be a maximum independent set in
$D$ and set $\gamma \DefinedAs 0$.  Otherwise $D$ contains an induced $P_3$
$abc$ and we let $J=\set{a, c}$ and set $\gamma \DefinedAs 1$.  (In this case $J$
is a maximum independent set, since $\alpha(D)= 2$.)

Lemma~\ref{BasicZeta} implies that
$\sum_{v\in J}d_D(v)\ge |D|-|J|+\gamma$.  Since $L$ is bad, we must have 
$\card{B}\ge \card{Pot(L')} \ge \sum_{v \in J} \card{L'(v)} 
\ge 2\card{J} + \sum_{v \in J} d_D(v) 
\ge 2\card{J} + \card{D} - \card{J} + \gamma 
\ge \card{J} + \card{D} + \gamma 
\ge \card{J} + \card{B} - 2 + \gamma$.

Taking the first and last expressions in the chain of inequalities gives
$\card{J} \leq 2 - \gamma$, a contradiction.  
Therefore $D$ is indeed the disjoint union of at most two complete subgraphs.  
We now consider the case that $D$ is a complete graph and the case that $D$ is
the disjoint union of two complete graphs.

First, suppose $D$ is a complete graph.  Now $\card{D} \geq 2$, since $B$ is not
almost complete. Put $X \DefinedAs N(x) \cap V(D)$ and $Y \DefinedAs N(y) \cap
V(D)$. Suppose $X - Y \neq \emptyset$ and pick $z \in X - Y$.  We have
$\card{L(y)} + \card{L(z)} \ge
d(y) + d(z) - 2 = d_B(y) + d_B(z) + 4 \geq 0 + \card{B} - 2 + 4 = \card{B} +
2 > \card{Pot(L)}$. By repeating the argument given above for $B-\set{x,y}$,
we see that $B - \set{y, z}$ is also the disjoint union of at most two
complete subgraphs.  In particular, $x$ is adjacent to all or none of $D - z$. 
If all, then $B$ is almost complete, if none then $B$ contains an induced $P_4$
or antipaw, and both possibilities give contradictions by Lemmas \ref{AJoinP_4}
and \ref{AntiPaw}.  Hence $X - Y = \emptyset$.  

By exchanging $x$ and $y$ in the argument above, $Y - X = \emptyset$,
so $X = Y$.  Since $B$ is not $\djunion{E_2}{K_{\card{B} - 2}}$, $\card{X} >
0$.  If $X = V(D)$, then $B$ is almost complete.  If $|V(D)-X|\ge 2$, then pick
$w_1, w_2\in V(D)-X$.  Now by considering degrees, we see that $L(x)\cap
L(w_1)$ and $L(y)\cap L(w_2)$ are both nonempty.  Now we can color $x, y, w_1,
w_2$ using only 2 colors, and then complete the coloring.  Hence, we must have
$|V(D)-X|=1$, so let $\{w\}=V(D)-X$.  Now $x$ and $y$ are joined to $D - w$ and
hence $B$ is $\join{E_3}{K_{\card{B} - 3}}$, a contradiction.

Thus $D$ must instead be the disjoint union of two complete subgraphs $D_1$ and
$D_2$.  For each $i \in \{1,2\}$, put $X_i \DefinedAs N(x) \cap V(D_i)$ and
$Y_i \DefinedAs N(y) \cap V(D_i)$.  

We know that $X_i \cap Y_i = \emptyset$, since otherwise we get a contradiction
as above when $\gamma = 1$.  
Suppose we have $z_1 \in V(D_1)$ and
$z_2 \in V(D_2)$ such that $L(z_1) \cap L(z_2) \neq \emptyset$. We must have $L(z_1) \cap L(z_2) = L(x) \cap L(y)$, for otherwise
we color each pair with a different color and finish greedily. Since no independent
set of size three can have a color in common, the edges $z_1x$ and $z_2y$ or
$z_1y$ and $z_2x$ must be present.
Using the same argument as for $B-\set{x,y}$, we see that $B - \set{z_1, z_2}$
is the disjoint union of at most two complete subgraphs.  
So each of $x$ and $y$ is adjacent to all or none of each of $V(D_1-z_1)$ and
$V(D_2-z_2)$.
Thus, by symmetry, we may assume that $V(D_1 - z_1) \subseteq X_1$
and $V(D_2 - z_2) \subseteq Y_2$.  

If $\card{D_1} = \card{D_2} = 1$, then $B$
is the disjoint union of two cliques, a contradiction.  So, by symmetry, we may
assume that $\card{D_1} \geq 2$.  Pick $w \in V(D_1 - z_1)$. If $x$ is not
adjacent to $z_1$, then $xwz_1$ is an induced $P_3$ in $B$.  Since $X_1 \cap
Y_1 = \emptyset$, this $P_3$ together with $y$ either induces a $P_4$ or an
antipaw, contradicting Lemmas \ref{AJoinP_4} and \ref{AntiPaw}.  Hence $X_1 =
V(D_1)$.  Similarly, if $\card{D_2} \geq 2$, then $Y_2 = V(D_2)$ and $B$ is the
disjoint union of two complete subgraphs, a contradiction.  Hence $D_2 =
\set{z_2}$.  But $z_2$ must be adjacent to $y$, so $B$ is again the disjoint
union of two cliques, a contradiction.

Thus for every $z_1 \in V(D_1)$ and $z_2 \in V(D_2)$ we have $L(z_1) \cap L(z_2) = \emptyset$. 
Suppose there exist $z_1\in V(D_1)$ and $z_2\in V(D_2)$ such that $z_1$ and
$z_2$ are each adjacent to at least one of $x$ and $y$.
Then $\card{L(z_1)} + \card{L(z_2)} \geq d(z_1) + d(z_2) - 2 \geq d_B(z_1) + d_B(z_2) + 4 \geq \card{B} - 4 + 2 + 4 = \card{B} + 2 > \card{Pot(L)}$.  Hence $L(z_1) \cap L(z_2) \neq \emptyset$, a contradiction.

Thus, by symmetry, we may assume that there are no edges between $D_1$ and
$\set{x, y}$.  Since no vertex in $D_2$ is adjacent to both $x$ and $y$, only
one of $x$ or $y$ can have neighbors in $D_2$ lest $B$ contain an induced $P_4$
contradicting Lemma \ref{AJoinP_4}.  Without loss of generality, we may assume
that $y$ has no neighbors in $D_2$. 

Suppose that $|D_1|\ge 2$, $|D_2|\ge 2$, and there exists $t\in D_2$ such that
$x$ and $t$ are nonadjacent.  Now choose $u,v\in V(D_1)$ and $w\in
V(D_2)\setminus\{t\}$.  
Now $\set{v, w, y}$ is independent and $\card{L(v)} + \card{L(w)} + \card{L(y}) \geq d(v) + d(w) + d(y) - 3 \geq d_B(v) + d_B(w) + d_B(y) + 6 \geq \card{B} + 2 > \card{Pot(L)}$. Hence either $L(v) \cap L(y) \neq \emptyset$ or $L(w) \cap L(y) \neq \emptyset$.  
Similarly, either $L(u)\cap L(x) \neq \emptyset$ or $L(t)\cap L(x)\neq
\emptyset$.  Thus, we can color 4 vertices using only 2 colors, and we can
complete the coloring.
So now either $|D_1|=1$, $|D_2|=1$, or $D_2\subset N(x)$.

If $|D_2|=1$, then either $B=K_1+K_2+K_{|B|-3}$ or else $B=E_3+K_{|B|-3}$, both
of which are forbidden.  Similarly, if $|D_1|=1$ and $x$ is adjacent to all or
none of $D_2$, then $B=K_1+K_1+K_{|B|-2}$ or $E_3+K_{|B|-3}$.
Finally, if $x$ is adjacent to some, but not all of $D_2$, then $B$ contains an
antipaw.  By Lemma~\ref{AntiPaw}, this is a contradiction.
\end{proof}

\subsection{Joins with $K_2$ and $E_2$}
\label{JoinsWithTwos}
\begin{lem}\label{ConnectedIncompleteAtLeast4}
$\join{A}{E_2}$ is $d_1$-choosable if $A$ is connected, incomplete and $|A| \ge 4$.
\end{lem}
\begin{proof}
Suppose not and let $L$ be a bad $d_1$-assignment on $\join{A}{E_2}$ with $|Pot(L)|$ minimum. 
By the Small Pot Lemma we know $\card{Pot(L)} \le |A| + 1$. Let $x_1$ and $x_2$
be the vertices in $E_2$.  Then $\card{L(x_1)} + \card{L(x_2)} = d(x_1) +
d(x_2) - 2 = 2\card{A} - 2 > \card{Pot(L)}$ since $|A| \ge 4$.  Hence we can
color $x_1$ and $x_2$ the same and Lemma \ref{ConnectedPot} shows that
$\card{Pot(L)} \le |A|$.  Hence $x_1$ and $x_2$ have at least $2\card{A} - 2 -
|A| = |A| - 2$ colors in common.  Each such color $c$ must appear in every list
in $A$, for otherwise we color $x_1$ and $x_2$ with $c$ and then greedily color
toward the $y \in A$ with $c \not \in L(y)$.

If $\alpha(A) \geq 3$, then coloring a maximum independent set all with the same color leaves an easily completable list assignment.  Also, if $A$ contains two disjoint pairs of nonadjacent vertices, then we may color each pair with a different color (there are at least $|A| - 2 \ge 2$ colors appearing on all vertices) and again complete the coloring.  Hence $A$ is almost complete.  

Let $c_1$ and $c_2$ be different colors appearing on all vertices. Choose $z \in V(A)$ such that $A - z$ is complete.  Since $A$ is incomplete, we have $w \in V(A-z)$ nonadjacent to $z$.  Also, as $A$ is connected we have $w' \in V(A-z)$ adjacent to $z$.  Color $x_1$ and $x_2$ with $c_1$ and $w$ and $z$ with $c_2$ to get a list assignment $L'$ on $D \DefinedAs A - \set{w, z}$ where
$\card{L'(v)} \geq d_D(v)$ for all $v \in V(D)$ and $\card{L'(w')} > d_D(w')$.
Now the coloring can be completed, a contradiction.
\end{proof}

\begin{lem}\label{E2Join2P3}
$\join{E_2}{2P_3}$ is $d_1$-choosable.
\end{lem}
\begin{proof}
Suppose otherwise. Let the $E_2$ have vertices $x_1$ and $x_2$ and the two
$P_3$s have vertices $y_1, y_2, y_3$ and $y_4, y_5, y_6$.  By the Small Pot
Lemma, we have a bad $d_1$-assignment on $\join{E_2}{2P_3}$ with
$\card{Pot(L)} \leq 7$. Since $\card{L(x_1)} + \card{L(x_2)} = 10 \geq
\card{Pot(L)} + 3$, we have three different colors $c_1, c_2, c_3 \in L(x_1)
\cap L(x_2)$.  Coloring both $x_1$ and $x_2$ with any $c_i$ leaves at worst a
$d_0$-assignment on the $2P_3$.  If $c_i \not \in L(y_1) \cap L(y_2) \cap
L(y_3)$ and $c_i \not \in L(y_4) \cap L(y_5) \cap L(y_6)$ for some $i$, then we
can complete the coloring.  Thus, without loss of generality, we have
$\set{c_1, c_2} \subseteq  L(y_1) \cap L(y_2) \cap L(y_3)$ and $c_3 \in
L(y_4) \cap L(y_5) \cap L(y_6)$.  Color $y_1$ and $y_3$ with $c_1$ and $y_4$
and $y_6$ with $c_3$.  Then we can easily complete the coloring on the rest of
the $2P_3$.  We have used at most $4$ colors on the $2P_3$ and hence we can
complete the coloring.
\end{proof}

\begin{lem}\label{E2Classification}
$\join{E_2}{B}$ is $d_1$-choosable if
$B$ is not the disjoint union of complete subgraphs and at most one $P_3$.  
\end{lem}
\begin{proof}
Suppose we have $B$ such that $\join{E_2}{B}$ is not $d_1$-choosable. By Lemma
\ref{E2Join2P3}, $B$ has at most one incomplete component.  Suppose we have an
incomplete component $C$ and let $y_1y_2y_3$ be an induced $P_3$ in $C$.  
If $C \neq P_3$, then $\card{C} \geq 4$ and Lemma \ref{ConnectedIncompleteAtLeast4} gives a contradiction.  Hence $C = P_3$.
\end{proof}

\input{chair_antichair}

\begin{lem}\label{MinimalK2}
$\join{K_2}{C_4}$ and $\join{K_2}{\text{antichair}}$ are $d_1$-choosable.
\end{lem}
\begin{proof}
$\join{K_2}{C_4} = \join{\join{E_2}{E_2}}{K_2}$ which is $d_1$-choosable by Lemma \ref{ConnectedIncompleteAtLeast4}.

Suppose $G \DefinedAs \join{K_2}{\text{antichair}}$ is not $d_1$-choosable. We use the labeling of the antichair given in Figure
\ref{fig:antichair}.  Let $L$ be a bad $d_1$-assignment on $G$ with $|Pot(L)|$ minimum.  By the Small Pot Lemma, we have $\card{Pot(L)} \le 6$.  
We have $\card{L(y_2)} + \card{L(y_5)} = d(y_2) + d(y_5) - 2 = 5 + 4 - 2 = 7 > \card{Pot(L)}$.  Hence we may color $y_2$ and $y_5$ the same and finish coloring the antichair greedily.  
By Lemma \ref{ConnectedPot}, we conclude $\card{Pot(L)} \le 5$.  Now $\card{L(y_1)} + \card{L(y_3)} = d(y_1) + d(y_3) - 2 = 3 + 5 - 2 = 6 > \card{Pot(L)}$.  Since $y_2$ and $y_5$ have at least two colors in common
and $y_1$ and $y_3$ have at least one color in common, we can color the pairs with different colors and complete the coloring greedily to all of $G$, a contradiction.
\end{proof}

%% file: antipaw.tex
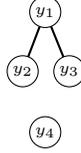
\begin{figure}[htb]
\centering
\begin{tikzpicture}[scale = 10]
\tikzstyle{VertexStyle}=[shape = circle,	
								 minimum size = 1pt,
								 inner sep = 1.2pt,
                         draw]
\Vertex[x = 0.481999963521957, y = 0.880285702645779, L = \tiny {$y_1$}]{v0}
\Vertex[x = 0.451999992132187, y = 0.801428586244583, L = \tiny {$y_2$}]{v1}
\Vertex[x = 0.513999998569489, y = 0.801428586244583, L = \tiny {$y_3$}]{v2}
\Vertex[x = 0.482000023126602, y = 0.720285713672638, L = \tiny {$y_4$}]{v3}
\Edge[](v1)(v0)
\Edge[](v2)(v0)
\end{tikzpicture}
\caption{The antipaw.}
\label{fig:antipaw}
\end{figure}

%% file: chair_antichair.tex
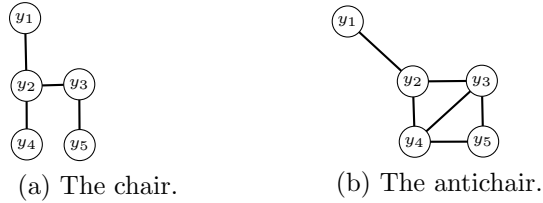
\begin{figure}[htb]
\centering
\subfloat[The chair.]{\label{fig:chair}
{\parbox{2.5cm}{
\begin{tikzpicture}[scale = 10]
\tikzstyle{VertexStyle}=[shape = circle,	
								 minimum size = 1pt,
								 inner sep = 1.2pt,
                         draw]
\Vertex[x = 0.236114323139191, y = 0.890571415424347, L = \tiny {$y_1$}]{v0}
\Vertex[x = 0.238114297389984, y = 0.800571441650391, L = \tiny {$y_2$}]{v1}
\Vertex[x = 0.308114320039749, y = 0.802571445703506, L = \tiny {$y_3$}]{v2}
\Vertex[x = 0.238400012254715, y = 0.722571432590485, L = \tiny {$y_4$}]{v3}
\Vertex[x = 0.308685749769211, y = 0.721714287996292, L = \tiny {$y_5$}]{v4}
\Edge[](v1)(v0)
\Edge[](v1)(v3)
\Edge[](v1)(v2)
\Edge[](v4)(v2)
\end{tikzpicture}}}}\qquad\qquad
\subfloat[The antichair.]{\label{fig:antichair}
{\parbox{3cm}{
\begin{tikzpicture}[scale = 10]
\tikzstyle{VertexStyle}=[shape = circle,	
								 minimum size = 1pt,
								 inner sep = 1.2pt,
                         draw]
\Vertex[x = 0.151542887091637, y = 0.880285702645779, L = \tiny {$y_1$}]{v0}
\Vertex[x = 0.238114297389984, y = 0.800571441650391, L = \tiny {$y_2$}]{v1}
\Vertex[x = 0.32982861995697, y = 0.801428586244583, L = \tiny {$y_3$}]{v2}
\Vertex[x = 0.240685701370239, y = 0.720285713672638, L = \tiny {$y_4$}]{v3}
\Vertex[x = 0.331542909145355, y = 0.720571428537369, L = \tiny {$y_5$}]{v4}
\Edge[](v1)(v0)
\Edge[](v1)(v3)
\Edge[](v1)(v2)
\Edge[](v4)(v2)
\Edge[](v4)(v3)
\Edge[](v3)(v2)
\end{tikzpicture}}}}
\caption{Labelings of the chair and the antichair.}
\label{fig:chair_antichair}
\end{figure}